\documentclass[11pt]{article}

 \usepackage{enumerate}
\usepackage{mathrsfs}
\usepackage{amsmath, color}
\usepackage{latexsym}
\usepackage{amssymb}
\usepackage{amsthm}
\usepackage{amsfonts}
\usepackage{dsfont}
\usepackage{hyperref}
\usepackage{xcolor,cancel}

\newtheorem{dfn}{Definition}[section]
\newtheorem{thm}[dfn]{Theorem}
\newtheorem{rmk}[dfn]{Remark}

\newtheorem{prop}[dfn]{Proposition}

\newtheorem{lem}[dfn]{Lemma}
\theoremstyle{definition}

\numberwithin{equation}{section}

\newcommand{\IE}{{\mathbb{E}}}
\newcommand{\IP}{{\mathbb{P}}}
\newcommand{\IR}{{\mathbb{R}}}
\newcommand{\FF}{{\mathcal{F}}}
\newcommand{\EE}{{\mathcal{E}}}
\newcommand{\LL}{{\mathcal{L}}}

\def\eps{\varepsilon}
\def\wh{\widehat}
\def\wt{\widetilde}
\def\<{\langle}
\def\>{\rangle}
\def\loc{{\rm loc}}

\title{\bf  Brownian Motion on Spaces with Varying Dimension
\thanks{Research partially supported
by NSF Grant DMS-1206276.}}
\date{\today}
\author{{\bf Zhen-Qing Chen} \quad and \quad {\bf Shuwen Lou}}
\begin{document}

\maketitle

\begin{abstract}
In this paper we introduce and study  Brownian motion on state spaces with varying dimension. Starting with a concrete case of such state spaces that models a big square with a flag pole, we construct a Brownian motion on it
 and study how heat propagates on such a space. 
 We derive sharp two-sided global estimates on its transition density functions (also called heat kernel). 
 These two-sided estimates are of Gaussian type,  but the measure on the 
 underlying state space does not satisfy volume doubling
property.
 Parabolic Harnack inequality fails for such a process.
 Nevertheless, we show H\"{o}lder regularity  holds for its parabolic functions. 
    We also derive the Green function estimates for this process on bounded smooth domains. 
    Brownian motion  
on some other state spaces with varying dimension are also constructed and studied in  this paper.

\end{abstract}



\section{Introduction}\label{Intro}

Brownian motion takes a  central place in  modern probability theory and its applications, 
and is a basic building block
for modeling many random phenomena. Brownian motion has intimate 
connections to analysis since its infinitesimal generator
is Laplacian operator. Brownian motion in Euclidean spaces has been  
studied by many authors in depth. 
Brownian motion on manifolds and on fractals 
has also been investigated vigorously, and is shown to have intrinsic interplay with the geometry of
the underlying spaces. See \cite{H, KS, MP, RY} and the references therein. 
In most of these studies,  the underlying metric measure spaces are assumed to satisfy volume doubling (VD) property. 
For Brownian  motion on manifolds with walk dimension 2, a remarkable fundamental result obtained independently by Grigor'yan \cite{Gr} and 
Saloff-Coste \cite{LSC1} asserts that the following are equivalent: (i)  two-sided Aronson type Gaussian bounds for heat kernel,
(ii) parabolic Harnack equality, and (iii)  VD and Poincar\'e inequality.  
This result is then extended to strongly local Dirichlet forms on metric measure space in \cite{BM, St1, St2}
and to graphs in \cite{De}.
For Brownian motion on fractals with walk dimension larger than 2,  
the above equivalence  still holds  but one needs to replace (iii)  
with (iii') VD, Poincar\'e inequality and a cut-off Sobolev inequality; see \cite{BB3, BBK, AB}.

Recently, analysis on non-smooth spaces has attracted lots of interest.
In real world, there are many objects having varying dimension. 
It is natural to study Brownian motion and ``Laplacian operator" on such spaces.
A simple example of spaces with varying dimension is a large square with a thin
flag pole. Mathematically, it is modeled by a plane with a vertical line installed 
on it:
\begin{equation}\label{e:1.1a}
\IR^2 \cup \IR_+: =\left\{(x_1, x_2, x_3)\in \IR^3: x_3=0 \textrm{ or } x_1=x_2=0 \hbox{ and } x_3>0\right\}. 
\end{equation}
 Here and in the sequel, we use $:=$ as a way of definition
 and   
denote $[0, \infty)$ by $\IR_+$. 
Spaces with varying dimension arise  in  many disciplines including statistics, physics  and engineering 
(e.g.  molecular dynamics, plasma dynamics).
See, for example,  \cite{LS, Xia} and the references therein.

The goal of this paper is to construct and study Brownian motion and Laplacian on   spaces
of varying dimension, in particular, to investigate how heat propagates on such spaces. 
 Intuitively, Brownian motion on space $\IR^2\cup \IR$ of  \eqref{e:1.1a}  should behave like a two-dimensional
Brownian motion when it is on the plane, and like a one-dimensional Brownian motion
when it is on the vertical line (flag pole).  However  the space $\IR^2\cup \IR$  is quite singular in the sense
that the base $O$ of the flag pole where the plane and the vertical line meet 
is a singleton. A singleton
would never be visited by a two-dimensional Brownian motion, which means 
 Brownian motion starting from a point on the plane will never visit $O$.
Hence there is no chance for such a process to climb up the flag pole. So the idea is to collapse or short (imagine putting
an infinite conductance on) a small closed disk $\overline{ B(0, \eps)}\subset \IR^2$ centered at the origin into a point $a^*$
 and   consider the resulting
Brownian motion with darning on the collapsed plane,
for which $a^*$ will be visited. The notion of Brownian motion with darning
is coined in \cite{CF} and its potential theory has been studied in details
 in \cite{Chen} and \cite[Sections 2-3]{CFR}.
Through $a^*$ we put a vertical  pole and   construct Brownian motion
with varying dimension on $\IR^2 \cup \IR_+$ by joining together
the Brownian motion with
darning on the plane and the one-dimensional Brownian motion along the pole.
It is possible to construct the process rigorously
via Poisson point process of excursions. But we find that
the most direct way to construct BMVD is by using
a Dirichlet form approach, which will be
carried out in Section 2.

To be more precise, the state space of BMVD on $E$ is defined as follows.
Fix $\eps>0$ and $p>0$.
Denote by  $B_\eps $  the closed disk on $\IR^2$ centered at $(0,0)$ with radius $\eps $. Let $D_0=\IR^2\setminus  B_\eps $. By identifying $B_\eps $ with a singleton denoted by $a^*$, we can introduce a topological space $E:=D_0\cup \{a^*\}\cup \IR_+$, with the origin of $\IR_+$ identified with $a^*$ and
a neighborhood of $a^*$ defined as $\{a^*\}\cup \left(U_1\cap \IR_+ \right)\cup \left(U_2\cap D_0\right)$ for some neighborhood $U_1$ of $0$ in $\IR^1$ and $U_2$ of $B_\eps $ in $\IR^2$. Let $m_p$ be the measure on $E$ whose restriction on $\IR_+$ and $D_0$ is the Lebesgue measure multiplied by $p$ and $1$, respectively.
In particular, we have $m_p (\{a^*\} )=0$.

\medskip

\begin{dfn}\label{D:1.1} \rm  Let $\eps >0$ and $p>0$. A Brownian motion with varying dimensions (BMVD in abbreviation) with parameters $(\eps, p)$ on $E$  is an $m_p$-symmetric diffusion $X$ on $E$ such that
\begin{description}
\item{(i)} its part process in $\IR$ or $D_0$ has the same law as standard Brownian motion in $\IR$ or $D_0$;
\item{(ii)} it admits no killings on $a^*$.
\end{description}
\end{dfn}

It follows from the $m_p$-symmetry of $X$ and the fact  $m_p(\{a^*\})=0$
that BMVD $X$ spends zero Lebesgue amount of time at $a^*$.

\medskip

The following  will be established in Section 2.

\begin{thm}\label{T:1.2}
 For every $\eps>0$ and $p>0$, BMVD with parameters $(\eps, p)$
exists and is unique in law.
\end{thm}

We point out that BMVD on $E$ can start from every point in $E$.
 We further characterize the $L^2$-infinitesimal generator $\LL$
 of BMVD $X$ in Section 2,     which can be viewed as the Laplacian operator on this singular
space.  We show that $u\in L^2(E; m_p)$
  is in the domain of the generator $\LL$ if and only if
  $\Delta u$ exists as an $L^2$-integrable function
  in the distributional sense when restricted to $D_0$ and $\IR_+$,
  and $u$ satisfies zero-flux condition at $a^*$; see Theorem \ref{T:2.3}
  for details. It is not difficult to see that BMVD $X$ has a continuous transition
  density function $p(t, x, y)$ with respect to the measure $m_p$, which is
  also called the fundamental solution (or heat kernel) for $\mathcal{L}$.
  Note that $p(t, x, y)$ is symmetric in $x$ and $y$.
  The main purpose of this paper is to investigate how the BMVD $X$ propagates in $E$;
  that is, starting from $x\in E$, how likely $X_t$ travels to position $y\in E$
  at time $t$. This amounts to study the properties of  $p(t, x, y)$ of $X$.
  In this paper, we will   establish the following sharp two-sided estimates on $p(t,x,y)$
  in Theorem \ref{T:smalltime} and Theorem \ref{largetime}.
  To state the results, we need first to introduce some notations.
Throughout  this paper, we will denote the geodesic metric on $E$ by $\rho$.
Namely, for $ x,y\in E$, $\rho (x, y)$ is the shortest path distance (induced
from the Euclidean space) in $E$ between $x$ and $y$.
For notational simplicity, we write $|x|_\rho$ for $\rho (x, a^*)$.
We use $| \cdot |$ to denote the usual Euclidean norm. For example, for $x,y\in D_0$,
$|x-y| $ is  the Euclidean distance between $x$ and $y$ in $\IR^2$.
Note that for $x\in D_0$, $|x|_\rho =|x|-\eps$.
 Apparently,
\begin{equation}\label{e:1.1}
\rho (x, y)=|x-y|\wedge \left( |x|_\rho + |y|_\rho \right)
\quad \hbox{for } x,y\in D_0
\end{equation}
and $\rho (x, y)= |x|+|y|-\eps$ when $x\in \IR_+$ and $y\in D_0$ or vice versa.
Here and in the rest of this paper, for $a,b\in \IR$, $a\wedge b:=\min\{a,b\}$.

The following are the main results of this paper.

\begin{thm}\label{T:smalltime}
 Let $T>0$ be fixed. There exist positive constants $C_i$, $1\le i\le 14$ so  that the transition density $p(t,x,y)$ of BMVD satisfies the following estimates when $t\in (0, T]$:
\begin{description}
\item{\rm (i)} For $x \in \IR_+$ and $y\in E$, 
\begin{equation}\label{stuff1}
\frac{C_1}{\sqrt{t}}e^{-C_2\rho (x, y)^2/t }  \le p(t,x,y)\le\frac{C_3}{\sqrt{t}}e^{- C_4\rho (x, y)^2/t}.
\end{equation}

\item{\rm (ii)} For $x,y\in D_0\cup \{a^*\}$, when $|x|_\rho+|y|_\rho<1$,
\begin{eqnarray}\label{stuff3}
&& \frac{C_{5}}{\sqrt{t}}e^{- C_{6}\rho(x,y)^2/t}+\frac{C_{5}}{t}\left(1\wedge
\frac{|x|_\rho}{\sqrt{t}}\right)\left(1\wedge
\frac{|y|_\rho}{\sqrt{t}}\right)e^{- C_{7}|x-y|^2/t} \\
 &\le  & p(t,x,y)
\le \frac{C_{8}}{\sqrt{t}}e^{-C_{9}\rho(x,y)^2/t}
+\frac{C_{8}}{t}\left(1\wedge
\frac{|x|_\rho}{\sqrt{t}}\right)\left(1\wedge
\frac{|y|_\rho}{\sqrt{t}}\right)e^{-C_{10}|x-y|^2/t}; \nonumber
\end{eqnarray}
and when $|x|_\rho+|y|_\rho\geq 1$,
\begin{equation}\label{stuff333333}
\frac{C_{11}}{t}e^{- C_{12}\rho(x,y)^2/t} \le p(t,x,y) \le \frac{C_{13}}{t}e^{- C_{14}\rho(x,y)^2/t}.
\end{equation}
 \end{description}
\end{thm}

Since $p(t, x, y)$ is symmetric in $(x, y)$, the above two cases cover all the cases for $x,y\in E$.
Theorem \ref{T:smalltime} shows that the transition density function $p(t, x, y)$ of BMVD on $E$
has one-dimensional character when at least one of $x, y$ is  in the pole 
(i.e. in $\IR_+$);
it has two-dimensional character when both points are on the plane and at least  one of them is away from the pole base $a^*$.
When both $x$ and $y$ are in the plane and both are close to $a^*$, 
$p(t, x, y)$ exhibits
a mixture of one-dimensional and two-dimensional characters; see \eqref{stuff3}.
Theorem \ref{T:smalltime} will be proved through Theorems \ref{T:4.5}-\ref{T:4.7}. 

The large time heat kernel estimates for BMVD are given by the next theorem, which are very different from the small time estimates.

\begin{thm}\label{largetime}
Let $T>0$ be fixed. There exist positive constants $C_i$, $15\le i \le 32$, so that the transition density $p(t,x,y)$ of BMVD satisfies the following estimates when $t\in [T, \infty)$:
\begin{enumerate}
\item[\rm (i)]  For $x,y\in D_0\cup \{a^*\}$,
\begin{equation*}
\frac{C_{15}}{t}e^{- C_{16}\rho(x,y)^2/t} \le p(t,x,y)\le \frac{C_{17}}{t}e^{- C_{18}\rho(x,y)^2/t}.
\end{equation*}

\item[\rm (ii)]  For $x\in \IR_+$, $y\in D_0\cup \{a^*\}$,
\begin{equation*}
\frac{C_{19}}{t}\left( 1 + \frac{|x| \log t}{\sqrt{t}} \right)
e^{- C_{20}\rho(x,y)^2/t} \le p(t,x,y)\le \frac{C_{21}}{t}
\left( 1 + \frac{|x| \log t}{\sqrt{t}} \right)
e^{-C_{22}\rho(x,y)^2/t}
\end{equation*}
when $|y|_\rho\le 1$, and
\begin{align*}
 &\frac{C_{23}}{t}  \left( 1 + \frac{|x|}{\sqrt{t}}  \log \left( 1+\frac{\sqrt{t}}{|y|_\rho} \right)   \right) 
 e^{-{C_{24}\rho(x,y)^2}/{t}} \le p(t,x,y)
\\
&\le \frac{C_{25}}{t} \left( 1 + \frac{|x|}{\sqrt{t}}  \log \left( 1+\frac{\sqrt{t}}{|y|_\rho} \right)   \right)    
e^{-{C_{26}\rho(x,y)^2}/{t}} \qquad \hbox{when }  |y|_\rho>1.
\end{align*}

\item[\rm (iii)] For $x,y\in \IR_+$,
\begin{eqnarray*}
&&\frac{C_{27}}{\sqrt{t}}\left(1\wedge \frac{|x|}{\sqrt{t}}\right)\left(1\wedge \frac{|y|}{\sqrt{t}}\right)e^{-C_{28}|x-y|^2/t}
+\frac{C_{27}}{t} \left( 1+ \frac{(|x|+|y|)\log t}{\sqrt{t}}\right) e^{-{C_{29}(x^2+y^2)}/{t}}
\\
&\le& p(t,x,y) \\
&\le & \frac{C_{30}}{\sqrt{t}}\left(1\wedge \frac{|x|}{\sqrt{t}}\right)\left(1\wedge \frac{|y|}{\sqrt{t}}\right)e^{-C_{31}|x-y|^2/t}
+\frac{C_{30}}{t} \left( 1+ \frac{(|x|+|y|)\log t}{\sqrt{t}}\right) e^{-{C_{32}(x^2+y^2)}/{t}} .
\end{eqnarray*}
\end{enumerate}
\end{thm}

Theorem \ref{largetime} will be proved through  Theorems \ref{154}, \ref{T:5.15} and \ref{T:5.17}.

Due to the singular nature of the space, the standard Nash inequality and Davies method
for obtaining heat kernel upper bound do not give sharp bound for our BMVD.
We can not employ  either the methods in \cite{Gr, LSC1, BM, St1, St2} on obtaining heat kernel estimates through volume doubling 
and Poincar\'e inequality or the approach through parabolic Harnack inequality.  In fact,
$(E, m_p)$ does not have volume doubling property, and 
we will show the parabolic Harnack inequality fails for BMVD $X$; see Proposition \ref{P:2.1} and 
Remark \ref{rmkonsmalltimeHKE}(iii).
Hence a  new approach is needed to study the heat kernel
of BMVD. 
A key role is played by the ``signed radial process" of BMVD, which we can analyze
and derive its two-sided heat kernel estimates. From it, by exploring the rotational symmetry
of BMVD, we can obtain short time sharp two-sided heat kernel estimates  
 by the following observation.
The sample paths of BMVD starting at $x$ reach $y$ at time $t$  
in two possible ways:  with or without passing through $a^*$.
The probability of the first scenario is given exactly by the probability transition density function of killed Brownian motion in $D_0$.
The probability of the second scenario can be computed by employing the strong Markov property of BMVD
at the first hitting time of the pole base $a^*$ and reducing it to the signed radial process of BMVD, exploring the symmetry of BMVD starting from $a^*$.
The large time heat kernel estimates are more delicate. 
For large time estimate, the key is to obtain the correct on-diagonal estimate.
This is done through some delicate analysis of BMVD and Bessel process on the plane. 
As a corollary of the sharp two-sided heat kernel estimates, we find that the usual form
of the parabolic Harnack inequality fails for parabolic functions of BMVD.
Nevertheless,  we will show in Section \ref{S:6} that  joint
H\"older regularity holds for bounded  parabolic functions of $X$. 

Let $X^D$ be the part process of BMVD killed upon exiting a bounded connected $C^{1,1}$ open subset $D$ of $E$. Denote by $p_{_D}(t,x,y)$ its  transition density function. Using the Green function estimates and boundary Harnack inequality for absorbing  Brownian motion in Euclidean spaces, we can derive sharp two-sided estimates on Green function 
\begin{equation*}
G_{_D}(x,y):=\int_{0}^\infty p_{_D}(t,x,y)dt 
\end{equation*}
for BMVD in $D$. Recall that an open set $D\subset \IR^d$ is called to be $C^{1,1}$ if there exist a localization radius $R_0>0$ and a constant $\Lambda_0>0$ such that for every $z\in \partial D$, there exists a $C^{1,1}-$function $\phi=\phi_z: \IR^{d-1}\rightarrow \IR$ satisfying $\phi(0)=0$, $\nabla \phi(0)=(0,\dots, 0)$, $\|\nabla \phi\|_\infty \le \Lambda_0$, $|\nabla \phi (x)-\nabla \phi (z)|\le \Lambda |x-z|$ and an orthonormal coordinate system $CS_z:\, y=(y_1, \dots , y_d):=(\wt {y}, y_d)$ with its origin at $z$ such that
\begin{equation*}
B(z, R_0)\cap D =\{y\in B(0, R_0)\textrm{ in }CS_z: y_d>\phi (\wt {y})\}.
\end{equation*}
For the state space $E$, an open set $D\subset E$ will be called $C^{1,1}$ in $E$, if $D\cap (\IR_+ \setminus \{a^*\})$ is a $C^{1,1}$ open set in $\IR_+$, and $D\cap D_0$ is  a  $C^{1,1}$ open set in $\IR^2$.

\begin{thm}\label{greenfunctionestimate}
Suppose $D$ is a bounded $C^{1,1}$ domain of $E$ that contains $a^*$.
Let $G_{D} (x,y)$ be the Green function of BMVD $X$ killed upon exiting $D$.
Then for $x\not= y$ in $D$, we have 
\begin{equation*}
G_{_D}(x,y)\asymp\left\{
  \begin{array}{ll}
    \delta_D(x)\wedge \delta_D(y),
\quad  &
 x, y \in D\cap \IR_+   ; 
\\ \\
  \left(\delta_D(y)\wedge 1\right)\left(\delta_D\wedge 1\right)+\ln\left(1+\frac{\delta_{D\cap D_0}(x)\delta_{D\cap D_0}(y)}{|x-y|^2}\right),
&  x,  y\in D\cap D_0   ; 
\\ \\
   \delta_D(x) \delta_D(y),
& \ x\in D\cap \IR_+  ,  \, y\in D\cap D_0 .
  \end{array}
\right.
\end{equation*}
Here $\delta_D(x):={\rm dist}_{\rho} (x, \partial D):=\inf\{\rho (x, z): z\notin   D\}$ 
and $\delta_{D\cap D_0} (x):=\inf\{\rho (x, z): z\notin   D\cap D_0\}$.
\end{thm}

For two positive functions $f$ and $g$, $f\asymp g$ means that $f/g$ is bounded between two positive constants.
In the following, we will also use  notation $f\lesssim g$ (respectively, $f\gtrsim g$) to mean that there is some constant $c>0$ so that $f\leq c g$
(respectively, $f\geq c g$).

\medskip

The above space $E$  of varying dimension  is special.
It serves as a toy model for further study on Brownian motion on more general 
spaces of varying dimension.  Another two examples of  spaces of varying dimension and BMVD on them
are given and studied in Section 8 of this paper.  
Even for this toy model, several interesting and non-trivial phenomena have arisen.
The heat kernel estimates on spaces of varying dimension are  quite delicate.
They are of Gaussian type but they are not of the classical Aronson
Gaussian type.    
The different dimensionality is also reflected in the heat
kernel estimates for BMVD. Even when both points $x$ and $y$ are on the plane, the heat kernel
$p(t, x, y)$ exhibits both one-dimensional and 
two-dimensional characteristics depending
on whether both points are close to the base point $a^*$ or not. 
In addition, both the Euclidean distance $|x-y|$ and the geodesic distance $\rho (x, y)$
between the two points $x$ and $y$ play a role in the kernel estimates. 
 As far as we know, this is the first paper that is devoted to the detailed study of heat propagation on spaces of varying dimension
 and related potential theory. Our approach is mainly probabilistic. 
For other related work and approaches on
 Markov processes living on spaces with
possibly different dimensions, we refer the reader to \cite{ES, HK, Ha, Ku}  and the references therein.

 The rest of the paper is organized as follows. Section 2 gives a Dirichlet form construction and characterization of BMVD,
 as well as its infinitesimal generator. Nash-type inequality for $X$ is given in Section 3.
 In Section 4, we present small time heat kernel estimates for $X$, while the large time estimates are given in
 Section 5. H\"older continuity of parabolic functions of $X$ is established in Section 6.
 Section 7 is devoted to the two-sided sharp estimates for Green function of $X$ in bounded $C^{1,1}$ domains in $E$ that contain the pole base $a^*$.
 BMVD on a large square with multiple vertical flag poles or with an arch  are studied in Sections 8.

 For notational convenience, in this paper we set
\begin{equation}\label{e:1.6}
\overline{p}_D(t,x,y):=p(t,x,y)-p_D(t,x,y),
\end{equation}
where $D$ is a domain of $E$ and $p_D(t,x,y)$ is the transition density of the part process killed upon exiting $D$.
In other words, for any non-negative function $f\geq 0$ on $E$,
\begin{equation}\label{e:1.7}
 \int_E  \overline p_D (t, x, y)  f(y) m_p(dy) = \IE_x \left[ f(X_t); t\geq \tau_D \right],
\end{equation}
where $\tau_D:=\inf\{t\geq 0: X_t\notin D\}$.
Thus while $p_D(t,x,y)$ gives the probability density
that  BMVD starting from $x$ hits $y$ at time $t$ without exiting $D$, 
 $\overline{p}_D(t,x,y)$ is the probability density for BMVD starting from $x$ leaves $D$ before ending up at $y$ at time $t$.

We use $C^\infty_c(E)$ to denote the space of continuous functions with compact support in $E$ so that
there restriction to $D_0$ and $\IR_+$ are smooth on $\overline D_0$ and on $\IR_+$, respectively.
 We also follow the convention that in the statements of the theorems or propositions $C, C_1, \cdots$ denote positive constants, whereas in their proofs $c, c_1, \cdots$ denote positive constants whose exact value is unimportant and may change
 from line to line.

\section{Preliminaries}\label{S:2}

Throughout this paper, we denote the Brownian motion with varying dimension by $X$  and its state space   by $E$.
In this section, we will construct BMVD using Dirichlet form approach.
For the definition and basic properties of Dirichlet forms, including the relationship between Dirichlet form,
$L^2$-infinitesimal generator, resolvents and semigroups, we refer the reader to \cite{CF} and \cite{FOT}.

For a connected open set $D\subset \IR^d$,  $W^{1,2} (D)$ is
the collection of  $L^2(D; dx)$-integrable functions whose first order derivatives (in
the sense of distribution) exist and are also in $L^2(D; dx)$. Define
$$
\EE^0 (f, g) =\frac12 \int_{D} \nabla f(x) \cdot \nabla g(x) dx ,
 \qquad f, g\in W^{1,2}(D).
 $$
It is well known that when $D$ is smooth,
$(\EE^0, W^{1,2}(D)) $ is a regular Dirichlet form on $L^2 (\overline D; dx)$ and its associated Hunt process
is the (normally) reflected Brownian motion on $\overline D$.
Moreover, every function $f$ in $W^{1,2}(D)$ admits a quasi-continuous version on $\overline D$, 
which will still be denoted by $f$.
A quasi-continuous function is defined quasi-everywhere (q.e. in abbreviation) on $\overline D$.
When $d=1$, by Cauchy-Schwartz inequality, every function in $W^{1,2}(D)$ is $1/2$-H\"older on $\overline D$.
Denote by $W^{1,2}_0 (D)$ the $\EE^0_1$-closure of $C^\infty_c(D)$, where $\EE^0_1(f, f):=\EE^0(u, u) + \int_D u(x)^2 dx$.
It is known that for any open set $D\subset \IR^d$,
$(\EE^0, W^{1,2}_0(D))$ is a regular Dirichlet form on $L^2(D; dx)$ associated with the
absorbing Brownian motion in $D$.

For a subset $A\subset E$, we define $\sigma_A:=\inf\{t>0, X_t\in A\}$ and  $\tau_{_A}:=\inf\{t\geq 0: X_t \notin A\}$.
 Similar notations will be used for other stochastic processes.
We will use $B_\rho(x,r)$ (resp. $B_e(x,r)$) to denote the 
 open ball in $E$ under the path metric $\rho$ 
(resp. in $\IR_+$ or $\IR^2$ under the Euclidean metric) centered at $x\in E$ with radius $r>0$.

A measure $\mu$ on $E$ is said to have  volume doubling property if
  there  exists a constant $C>0$
 so that $m_p (B_\rho (x , 2r))\leq C  m_p(B_\rho (x , r))$
 for all $x\in E$ and every $r\in (0, 1]$. 

\begin{prop}\label{P:2.1} For any $p>0$, volume doubling property fails for  measure $m_p$.  
\end{prop}

\begin{proof}
Note that  for small $r>0$ and $x_0 \in D_0$ with $|x_0|_\rho=r$,
$m_p(B_\rho (x_0, r))= \pi r^2$ while $m_p (B_\rho (x_0, 2r))
 = 2\eps r +r^2 +r$. Thus there does not exist a constant $C>0$
 so that $m_p (B_\rho (x , 2r))\leq C  m_p(B_\rho (x , r))$
 for all $x\in E$ and every $r\in (0, 1]$. 
\end{proof}

The following is an extended version of  Theorem \ref{T:1.2}.

\begin{thm}\label{existence-uniqueness}
For every $\eps >0$ and $p>0$,
BMVD $X$ on $E$ with parameter $(\eps, p)$ exists and is unique.
Its associated Dirichlet form $(\EE, \FF)$ on $L^2(E; m_p)$ is given by
\begin{eqnarray}
\FF &= &  \left\{f: f|_{D_0}\in W^{1,2}(D_0),  \, f|_{\IR_+}\in W^{1,2}(\IR_+),
\hbox{ and }
f (x) =f (0) \hbox{ q.e. on } \partial D_0 \right\},   \label{e:2.1}
\\
\EE(f,g) &=& \frac12 \int_{D_0}\nabla f(x) \cdot \nabla g(x) dx+\frac{p}2\int_{\IR_+}f'(x)g'(x)dx .  \label{e:2.2}
\end{eqnarray}
\end{thm}

\begin{proof}
Let $\EE$ and $\FF$ be defined as above.
Let $u_1(x)=\IE^x\left[e^{-\sigma_{B_\eps}}\right]$ when $x\in D_0$ and $u_1(x)=\IE^x\left[e^{-\sigma_{D_0}}\right]$ when $x\in \IR_+$.
It is known that $u_1|_{D_0}\in W^{1,2}(D_0)$, $u_1|_{\IR_+} \in W^{1,2}(\IR_+)$,  $u_1(x)=1$ for $x\in \partial D_0$ and $u_1(0)=1$.
Hence $u_1\in \FF$.

 {\it Existence:}    Let
$$
\FF^0  =  \left\{f: f|_{\IR^2}\in W^{1,2}_0(D_0),  \, f|_{\IR_+}\in W^{1,2}_0(\IR_+) \right\} .
$$
Then $\FF$  is the   linear span of $\FF^0 \cup \{u_1\}$.
It is easy to check that $(\EE, \FF)$ is a strongly local regular Dirichlet form on $L^2(E;m)$.
So there is an $m_p$-symmetric diffusion process $X$ on $E$ associated with it.
Using the Dirichlet form characterization, the part process of $X$ killed upon hitting $a^*$ is
an absorbing Brownian motion in $D_0$ or $\IR_+$, depending on the starting point.
So $X$ is a BMVD on $E$. Moreover, $X$ is conservative; that is, it has infinite lifetime.

\medskip

{\it Uniqueness:} Conversely, if $X$ is a BMVD, it suffices to check from definition that its associated Dirichlet form  $(\EE^*, \FF^*)$
in $L^2(E; m_p)$  has to be  $(\EE, \FF)$ given in \eqref{e:2.1}-\eqref{e:2.2}. Indeed, since $a^*$ is non-polar for $X$, for all $u\in \FF^*$, 
$$
H_{a^*}^1u(x):=\IE^x\left[e^{-\sigma_{a^*}}u(X_{\sigma_{a^*}})\right]=u(a^*)\IE^x[e^{-\sigma_{a^*}}]\in \FF \cap \FF^*
$$
 and $u-H_{a^*}^1u(x)\in \FF^0$.
Thus $\FF^*\subset \FF$.  On the other hand, since the part process of $X$ killed upon hitting $a^*$ has the same distribution
as the absorbing Brownian motion on $D_0\cup (0, \infty)$, which has Dirichlet form $(\EE, \FF^0)$ on $L^2(E\setminus \{a^*\}; m_p)$,
we have $\FF^0\subset \FF^*$.  It follows that $\FF\subset \FF^*$ and therefore $\FF =\FF^*$.
Since $X$ is a diffusion that admits no killings inside $E$, its Dirichlet form
 $(\EE^*, \FF^*)$ is strongly local.  Let $\mu_{\<u\>}=\mu^c_{\<u\>}$
 denote the energy measure associated with $u\in \FF^*$; see \cite{CF, FOT}. Then  by the strong locality of $\mu^c_{\<u\>}$
 and the $m_p$-symmetry of $X$,  we have for  every bounded $u\in \FF^*=\FF$,
\begin{eqnarray*}
\EE^*(u,u) &= & \frac12 \left( \mu^c_{\langle u\rangle}(E\setminus  \{a^*\})+\mu_{\langle u\rangle }^c(a^*) \right)
=\frac12 \mu_{\langle u\rangle }^c(E\setminus  \{a^*\})=
\frac12 \mu_{\langle u\rangle}^c(D_0)+ \frac12 \mu_{\langle u\rangle}^c(\IR_+) \\
&=& \frac12 \int_{D_0}|u(x)|^2dx+ \frac{p}2\int_{\IR_+}|u'(x)|^2dx = \EE (u, u).
\end{eqnarray*}
 This proves $(\EE^*, \FF^*)=(\EE, \FF)$.
\end{proof}

Following \cite{CF}, for any $u\in \FF$, we define its flux $\mathcal{N}_p(u)(a^*)$   at $a^*$ by
$$
\mathcal{N}_p(u)(a^*)=\int_E \nabla u(x)\cdot \nabla  u_1(x)m_p(dx)+\int_E\Delta u(x)u_1(x)m_p(dx),
$$
which by the Green-Gauss formula equals
$$
\int_{\partial B_\eps}\frac{\partial u(x)}{\partial \vec{n}}\sigma(dx)-pu'(a^*).
$$
Here  $\vec{n}$ is the unit inward normal vector field of $B_\eps$ at $\partial B_\eps$.
The last formula justifies the name of ``flux" at $a^*$.

Let ${\cal L}$ be  the $L^2$-infinitesimal generator of $(\EE, \FF)$, or equivalently, the BMVD $X$,
with domain of definition $\mathcal{D}(\mathcal{L})$.
It can be viewed as the ``Laplacian" on the space $E$ of varying dimension.

\begin{thm}\label{T:2.3}
A function $u\in \FF$ is in $\mathcal{D}(\mathcal{L})$ if and only if the distributional Laplacian $\Delta u$ of $u$ exists as an $L^2$-integrable function on $E\setminus  \{a^*\}$ and $u$ has zero flux at $a^*$. 
Moreover, 
$\mathcal{L}u=\frac12 \Delta u$ on $E\setminus  \{a^*\}$
for $u\in \mathcal{D}(\mathcal{L})$. 
\end{thm}

\begin{proof}
By the Dirichlet form characterization,  $u\in \mathcal{D}(\mathcal{L})$ if and only if $u\in \FF$ and there is some $f\in L^2(E; m_p)$ so that
\begin{equation*}
\EE(u, v)=-\int_E f(x)v(x)m_p(dx) \quad \textrm{for every }v\in \FF.
\end{equation*}
In this case,  $\mathcal{L}u := f$. The above is equivalent to
\begin{equation}\label{961240}
\frac12 \int_E\nabla u(x)\cdot \nabla v(x)m_p(dx) =-\int_E f(x)v(x)m_p(dx) \quad \textrm{for every }v\in C_c^\infty (E\setminus  \{a^*\})
\end{equation}
and
\begin{equation}\label{961241}
\frac12 \int_E\nabla u(x)\cdot \nabla u_1(x)m_p(dx)=-\int_E f(x)u_1(x)m_p(dx).
\end{equation}
Equation \eqref{961240} implies that $f=\frac12 \Delta u\in L^2(E; m_p)$,
and \eqref{961241} is equivalent to $\mathcal{N}_p(u)(a^*)=0$.
\end{proof}

\section{Nash  Inequality and Heat Kernel Upper Bound
Estimate} \label{S:3}

  Recall that $D_0 :=\IR^2\setminus  \overline{B_e(0, \eps)}$.
In the following,  if no measure is explicitly mentioned in the $L^p$-space, it is understood as being with respect to
the measure $m_p$; for instance,  $L^p(E)$ means $L^p(E; m_p)$.

\begin{lem}\label{NashIneq}There exists $C_1>0$ so that
\begin{displaymath}
\|f\|_{L^2(E)}^2\le
C_1\left(\EE(f,f)^{{1}/{2}}\|f\|_{L^1(E)}+\EE(f,f)^{\
{1}/{3}}\|f\|_{L^1(E)}^{{4}/{3}}\right)
\qquad \hbox{for every }   f\in \FF.
\end{displaymath}
\end{lem}

\begin{proof}
Since $D_0\subset \IR^2$ and $\IR_+$  are smooth domains, we have by the classical Nash's inequality,
 \begin{equation*}
\|f\|_{L^2(D_0)}^2\le c\|\nabla f\|_{L^2(D_0)}\,  \|f\|_{L^1(D_0)}
\quad \hbox{for }  f\in W^{1,2}(D_0) \cap L^1(D_0),
\end{equation*}
and
 $$
 \|f\|^3_{L^2(\IR_+)}\le C\|f\|^2_{L^1(\IR_+)}\|f'\|_{L^2(\IR_+)}
 \quad \hbox{for }   f\in W^{1,2}(\IR_+) \cap L^1(\IR_+) .
 $$
 The desired inequality now follows by combining these two inequalities.
\end{proof}

  The Nash-type inequality in Lemma \ref{NashIneq} immediately implies that BMVD $X$ on $E$ has a symmetric density function
  $p(t, x, y)$ with respect to the measure $m_p$ and that the following
on-diagonal estimate by \cite[Corollary 2.12]{CKS} holds.

\begin{prop}\label{P:3.2} 
There exists $C_2>0$ such that
\begin{displaymath}
p(t, x, y)\le
C_2\displaystyle{\left(\frac{1}{t}+\frac{1}{t^{1/2}}\right)} \qquad \hbox{for all }
t>0 \hbox{ and  } x, y\in E.
\end{displaymath}
\end{prop}

 Since $X$ moves like Brownian motion in Euclidean spaces
before hitting $a^*$, it is easy to verify that for each $t>0$,
 $(x, y) \mapsto p(t, x, y)$ is continuous in $(E\setminus \{a^* \}) \times (E\setminus \{a^* \})$.
 For each $t>0$ and fixed $y\in E\setminus \{a^*\}$,
 $$
  \int_E p(t/2, y, z)^2 m_p(dz) = p(t, y, y)<\infty.
 $$
 So by the Dirichlet form theory,
  $x\mapsto p(t, x, y)= \int_E p(t/2, x, z) p(t/2, z, y) m_p(dz)$ is $\EE$-quasi-continuous on $E$.
 Since $a^*$ is non-polar for $X$,  $x\mapsto p(t, x, y)$ is continuous at $a^*$, and hence
 is continuous on $E$. By the symmetry and Chapman-Kolmogorov equation again,
 we conclude that
 $$ x\mapsto p(t, x, a^*)= \int_{E } p(t/2, x, z) p(t/2, z, a^*) m_p(dz)
 $$
 is continuous on $E$. Consequently, $p(t, x, y)$ is well defined pointwisely on $(0, \infty) \times E\times E$
 so that for each fixed $t>0$ and $y\in E$, $p(t, x, y)$ is a continuous function in $x\in E$.
 We can use Davies method to get an off-diagonal upper bound estimate.

\begin{prop}\label{offdiagUBE}
There exist $C_3, C_4>0$ such that
\begin{displaymath}
p(t,x,y)\le C_3\left(\frac{1}{t}+
\frac{1}{t^{1/2}}\right)e^{-C_4 {\rho(x,y)^2}/{t}}
\qquad \hbox{for all }  t>0 \hbox{ and  } x, y\in E.
\end{displaymath}
\end{prop}

\begin{proof}
Fix $x_0, y_0\in E$, $t_0>0$. Set a constant $\alpha:=\rho(y_0,x_0)/4t_0$ and
$\displaystyle{\psi (x):=\alpha|x|_\rho}$.
Then we define $\psi_n(x)=\psi(x)\wedge n$. Note that for $m_p$-a.e. $x\in E$,
\begin{displaymath}
e^{-2\psi_n (x)}|\nabla e^{\psi_n (x)} |^2=|\nabla
\psi_n (x) |^2=|\alpha|^2\, \mathbf{1}_{\{|x|_\rho \le
\frac{n}{|\alpha|}\}}(x) \leq \alpha^2.
\end{displaymath}
Similarly,   $ e^{2\psi_n (x)}|\nabla  e^{-\psi_n (x)}|^2 \leq \alpha^2$.
By \cite[Corollary 3.28]{CKS},
\begin{equation}\label{proveoffdiagUB}
p(t,x,y)\le c\left(\frac{1}{t}+
\frac{1}{t^{1/2}}\right)\exp\left(-|\psi(y)-\psi(x)|+2t|\alpha|^2\right)  .
\end{equation}
  Taking $t=t_0, x=x_0$
and $y=y_0$ in \eqref{proveoffdiagUB} completes the proof.
\end{proof}

As we have seen from Theorems \ref{T:smalltime} and \ref{largetime},
the upper bound estimate in Proposition \ref{offdiagUBE} is not sharp.

\section{Signed Radial Process and Small Time Estimate}\label{S:4}

In order to get the sharp two-sided heat kernel estimates, we consider the radial process of $X$. Namely, we project $X$ to $\IR$ by applying the following
mapping from $E$ to $\IR$:
\begin{equation}\label{848}
u(x)=\left\{
  \begin{array}{ll}
    -|x|, & \hbox{$x \in \IR_+$;} \\ \\
   |x|_\rho, & \hbox{$x\in D_0$.}
  \end{array}
\right.
\end{equation}
We call $Y_t:=u(X_t)$ the signed radial process of $X$.
Observe that $u\in \FF_{\loc}$, where $\FF_{\loc}$ denotes the local Dirichlet space of $(\EE,  \FF)$,
whose definition can be found, for instance,
in \cite{CF, FOT}. By Fukushima decomposition  \cite[Chapter 5]{FOT},
\begin{displaymath}
Y_t-Y_0=u(X_t)-u(X_0)=M^{[u]}_t+N^{[u]}_t, \quad \IP_x\textrm{-a.s. for q.e.
}x\in E,
\end{displaymath}
where $M^{[u]}_t$ is a local martingale additive functional of $X$, and
$N^{[u]}_t$ is a continuous additive functional of $X$ locally having  zero energy.
We can explicitly compute $M^{[u]}$ and $N^{[u]}$.
 For any $\psi\in C_c^\infty(E)$,
\begin{align*}
 \EE(u, \psi)&=\frac12 \int_{D_0}\nabla |x| \cdot \nabla \psi dx + \frac{p}2\int_{\IR_+}(-1)\psi' dx
\\&=- \frac12 \int_{D_0}\textrm{div}\left(\frac{x}{|x|}\right)\psi dx
- \frac12 \int_{\partial B_e(0,\eps)}\psi(0)\frac{\partial |x| }{\partial
\vec{n}}   \sigma (dx) + \frac{p\psi(0)}2
\\&=- \frac12 \int_{D_0}\frac{1}{|x|}\psi dx- \frac{2\pi\eps- p }2\, \psi(0) \\
&= -  \int_E \psi (x) \nu (dx),
\end{align*}
where   $\vec{n}$ is the outward pointing unit vector normal of the
surface $\partial B_e(0,\eps)$, $\sigma$ is the surface measure
on $\partial B_e (0, \eps)\subset \IR^2$, and
  \begin{displaymath}
\nu(dx):= \frac{1}{|2x|}\mathbf{1}_{ D_0 } (x) dx+\frac{2\pi \eps-p}2 \, \delta_{\{a^*\}}.
\end{displaymath}
 Recall that we identify $0\in \IR_+$ with $a^*$. It follows from \cite[Theorem 5.5.5]{FOT} that
\begin{equation}\label{Zero-energy-for-radial}
dN_t^{[u]}= \frac{1}{2( u(X_t)+\eps )}\mathbf{1}_{\{X_t\in
D_0\}}dt+(2\pi \eps-p)dL^0_t (X),
\end{equation}
where $L^0_t (X)$ is the positive continuous additive functional of $X$ having Revuz measure
$\frac12 \delta_{\{a^*\}}$. We call $L^0$ the local time of $X$ at $a^*$.
 Next we compute $\<M^{[u]}\>$, the predictable quadratic variation process of local martingale $ M^{[u]}) $.
   Let $u_n=(-n)\vee u \wedge n$, and it immediately follows $u_n\in \FF$.  Let $\FF_b$ denote the space of bounded functions in $\FF$.
By \cite[Theorem 5.5.2]{FOT}, the Revuz measure $\mu_{\<u_n\>}$ for $\< M^{[u_n]}\>$ can be calculated as follows.
For any $f\in \FF_b\cap C_c(E)$,
\begin{eqnarray*}
\int_E f (x) \mu_{\<u_n\>} (dx)  =
 2\EE(u_n  f, u_n)-\EE(u_n^2, f)
=   \int_E f(x) |\nabla u_n (x)|^2  m_p (dx) ,
\end{eqnarray*}
which shows that
\begin{displaymath}
\mu_{\langle  u_n \rangle}(dx)= |\nabla
u_n (x)|^2 m_p(dx)=  \mathbf{1}_{B_\rho(a^*, n)} m_p (dx).
\end{displaymath}
By the strong local property of $(\EE, \FF)$, we have $\mu_{\<u\>}= \mu_{\langle  u_n \rangle}$
on $B_\rho (a^*, n)$. It follows that $\mu_{\langle  u \rangle}(dx) = m_p (dx)$.
Thus by \cite[Proposition 4.1.9]{CF},  $\<M^{[u]}\>_t=t$ for $t\geq 0$
and so $B_t:= M^{[u ]}_t$ is a one-dimensional Brownian motion.
  Combining this with \eqref{Zero-energy-for-radial}, we  conclude
\begin{eqnarray}\label{710901}
dY_t  &=& dB_t+\frac{1}{2 (Y_t+\eps) }\mathbf{1}_{\{X_t\in
D_0\}}dt+(2\pi \eps-p)dL^0_t(X) \nonumber \\
&=& dB_t+\frac{1}{2 (Y_t+\eps) }\mathbf{1}_{\{Y_t >0\}} dt
  +(2\pi \eps-p)dL^0_t(X)  .
\end{eqnarray}

We next find the SDE for the semimartingale $Y$.
The semi-martingale local time of $Y$ is denoted as $L_t^0(Y)$,
that is,
 \begin{equation}\label{e:4.3}
L_t^0(Y):=\lim_{\delta \downarrow 0}\frac{1}{\delta}\int_0^t 1_{[0, \delta )}(Y_s)d\langle Y\rangle_s
= \lim_{\delta \downarrow 0}\frac{1}{\delta }\int_0^t 1_{[0, \delta )}(Y_s)ds,
\end{equation}
 where $\< Y\>_t=t$ is the quadratic variation process of the semimartingale $Y$.

\begin{prop}\label{localtime}
  $L_t^0(Y)= 4\pi\eps L_t^0(X) $.
\end{prop}

\begin{proof}
 By computation analogous to that for $Y_t=u(X_t)$, 
one can derive using Fukushima's decomposition for  $v(X_t):=|X_t|_\rho$,
 that 
 \begin{displaymath}
dv(X_t)=d\widetilde{B_t}+\frac{1}{2( |X_t|_\rho +\eps )}\mathbf{1}_{\{X_t\in
D_0 \}}dt+(2\pi\eps+p)dL_t^0(X),
\end{displaymath}
where  $\widetilde{B}$ is a one-dimensional Brownian motion.
Observe that $v(X_t)=|Y_t|$. Thus we have
 \begin{displaymath}
d|Y_t|=d\widetilde{B_t}+\frac{1}{2( Y_t+\eps ) }\mathbf{1}_{\{Y_t>0\}}dt+(2\pi\eps+p)dL_t^0(X).
\end{displaymath}
On the other hand, by Tanaka's formula, we have
\begin{align*}
d|Y_t|&=\textrm{sgn}(Y_t)dY_t+dL_t^0(Y)\\
&=\textrm{sgn}(Y_t)dB_t+\textrm{sgn}(Y_t)\frac{1}{2( Y_t+\eps )}\mathbf{1}_{\{Y_t>0\}}dt
 +\textrm{sgn}(Y_t)(2\pi\eps-p)dL_t^0(X)+dL_t^0(Y)\\
&=\textrm{sgn}(Y_t)dB_t+\frac{1}{ 2(Y_t+\eps )}\mathbf{1}_{\{Y_t>0\}}dt+(2\pi\eps-p)\textrm{sgn}(Y_t)dL_t^0(X)+dL_t^0(Y),
\end{align*}
where $\textrm{sgn}(x):=1$ if $x>0$  
and $\textrm{sgn}(x):=- 1$ if $x\leq 0$.
Since the decomposition of a continuous semi-martingale as the
sum of a continuous local martingale and a continuous   process with finite
variation is unique,  one must have
\begin{equation}\label{e:4.4}
(2\pi\eps+p)L_t^0(X)=\textrm{sgn}(Y_t)(2\pi\eps-p)L_t^0(X)+L_t^0(Y).
\end{equation}
The local time $L_t^0(X)$ increases only when  $ Y_t=0 $.
Therefore
\begin{displaymath}
(2\pi\eps +p)L_t^0(X)=- (2\pi\eps -p)L_t^0(X)+L_t^0(Y),
\end{displaymath}
and so $ 4\pi\eps L_t^0(X)=L_t^0(Y)$.
\end{proof}

The semi-martingale local time in \eqref{e:4.3} is non-symmetric in the   sense that
 it only measures the occupation time of $Y_t$ in the one-sided interval $[0, \delta )$ instead of the symmetric interval $(-\delta, \delta)$.
One can always relate the non-symmetric semi-martingale local time $L^0(Y)$
to the symmetric semi-martingale local time $\widehat L^0(Y)$ defined by
\begin{equation*}
\widehat{L}_t^0(Y):=\lim_{\delta \downarrow 0}\frac{1}{2 \delta }\int_0^t 1_{(-\delta, \delta )}(Y_s)d\langle Y\rangle_s
= \lim_{\delta \downarrow 0}\frac{1}{2 \delta}\int_0^t 1_{(-\delta, \delta )}(Y_s) d s .
\end{equation*}

\begin{lem}\label{L:4.2}
 $\widehat{L}_t^0(Y)=   \frac{2\pi \eps + p}{4\pi \eps } L_t^0(Y) = (2\pi \eps + p) L^0_t (X)$.
\end{lem}

\begin{proof} Viewing $|Y_t| = |-Y_t|$ and applying Tanaka's formula to the semimartingale $-Y$, we can derive in a way analogous to the computation leading to \eqref{e:4.4} that
$$
(2\pi\eps+p)L_t^0(X)=- \textrm{sgn}(-Y_t)(2\pi\eps-p) L_t^0(X)+L_t^0(-Y)
= ( 2\pi\eps-p)L_t^0(X)+L_t^0(-Y).
$$
Thus we get   $ 2pL_t^0(X)=L_t^0(-Y)$,  which yields 
$$
\widehat{L}_t^0(Y) =\frac{1}{2}\left(L_t^0(Y)+L_t^0(-Y)\right)
=\frac{1}{2}\left(4\pi\eps L_t^0(X)+2pL_t^0(X)\right)  = (2\pi\eps+p)L_t^0(X).
$$
\end{proof}

Lemma \ref{L:4.2}  together with \eqref{710901} gives the following SDE characterization for the signed radial
process $Y$, which tells us precisely how $X$ moves after hitting $a^*$.

\begin{prop}\label{P:4.3}
\begin{equation}\label{YsymmSDE}
dY_t=dB_t+\frac{1}{2(Y_t+\eps)}\mathbf{1}_{\{Y_t>0\}}dt+\frac{2\pi\eps-p}{2\pi\eps+p}d\widehat{L}_t^0(Y).
\end{equation}
\end{prop}

\medskip

Let $\beta = \frac{2\pi\eps-p}{2\pi\eps+p}$.
SDE \eqref{YsymmSDE} says that $Y$ is a skew Brownian motion with drift on $\IR$ with skew parameter $\beta$.
It follows (see \cite{RY})  that
starting from $a^*$, the process $Y$ (respectively, $X$) has probability $(1-\beta)/2= \frac{p}{2\pi \eps +p}$ to enter $(-\infty, 0)$
(respectively, the pole) and probability $(1+\beta )/2=\frac{2\pi \eps}{2\pi \eps +p}$ to enter $(0, \infty)$ (respectively, the plane).

 \medskip

SDE \eqref{YsymmSDE} has a unique strong solution; see, e.g., \cite{BC}. So $Y$ is a strong Markov process on $\IR$.
The following is a key to get the two-sided sharp heat kernel estimate on $p(t, x, y)$ for BMVD $X$.

\begin{prop}\label{P:4.4}
The one-dimensional diffusion process $Y$ has a jointly continuous transition density function
$P^{(Y)}(t, x, y)$ with respect to the Lebesgue measure on $\IR$. Moreover,  for every $T\geq 1$,
there exist constants  $C_i>0$, $1\le i \le 4$, such that the following estimate holds:
\begin{equation}\label{e:4.5}
\frac{C_1}{\sqrt{t}} e^{- {C_2|x-y|^2}/{t}} \le p^{(Y)}(t,x,y)\le\frac{C_3}{\sqrt{t}} e^{- {C_4|x-y|^2/}/{t}},
\quad (t,x,y)\in (0,T]\times \IR\times \IR.
\end{equation}
\end{prop}

\proof  Let  $\beta:=\frac{2\pi\eps-p}{2\pi\eps+p}$ and $Z$ be the skew Brownian motion
$$ dZ_t = dB_t + \beta \wh L^0_t (Z) ,
$$
where $\wh L^0_t (Z)$ is the symmetric local time of $Z$ at $0$.
The diffusion process $Y$ can be obtained from $Z$ through a drift perturbation
(i.e. Girsanov transform).
The transition density function $p_0(t, x, y)$ of $Z$ is explicitly known and
enjoys the two-sided Aronson-type Gaussian estimates \eqref{e:4.5};
see, e.g., \cite{RY}. One can further verify  that
$$ | \nabla_x p_0(t, x, y)| \leq c_1  t^{-1} \exp (-c_2 |x-y|^2/t),
$$
from which one can deduce \eqref{e:4.5} by using the same argument as that for Theorem A in Zhang \cite[\S 4]{Z1}.
\qed

\medskip

Proposition \ref{P:4.4} immediately gives the two-sided estimates on the transition function
$p(t, x, y)$ of $X$ when $x, y\in \IR_+$ since $X_t=-Y_t$ when $X_t\in \IR_+$.

\medskip

\begin{thm} \label{T:4.5}
For every $T\geq 1$,
there exist $C_i>0$, $5\le i \le 8$, such that the following estimate holds:
\begin{equation*}
\frac{C_5}{\sqrt{t}} e^{-{C_6|x-y|^2}/{t}} \le p (t,x,y)\le\frac{C_7}{\sqrt{t}} e^{-{C_8|x-y|^2}/{t}}
\qquad  \hbox{for } t\in (0, T] \hbox{ and } x, y \in \IR_+ .
\end{equation*}
\end{thm}

\medskip

Let $A$ be any rotation of the plane around the pole.
Using the fact that starting from $a^*$, $AX_t$ has the same distribution as 
$X_t$,
we can derive estimates for $p(t, x, y)$ for other $x, y \in E$.
The next result gives the two-sided estimates on
$p(t,x,y)$ when $x\in \IR$ and $y\in D_0$.

\begin{thm}\label{T:4.6} For every $T\geq 1$,
there exist constants $C_i>0$, $9\le i\le 12$, such that for all  $x\in \IR_+$, $y\in D_0$ and $t\in [0,T]$,
\begin{equation*}
\frac{C_9}{\sqrt{t}}e^{-{C_{10}\rho(x,y)^2}/{t}} \le p(t,x,y)\le
\frac{C_{11}}{\sqrt{t}}e^{-{C_{12}\rho(x,y)^2}/{t}}.
\end{equation*}
\end{thm}

\begin{proof}
We first note that in this case by the symmetry of $p(t, x, y)$,
\begin{equation*}
p (t,x,y) = p(t, y, x) =\int_{0}^t \IP_y(\sigma_{a^*}\in ds)p(t-s, a^*,x).
\end{equation*}
By the rotational invariance of two-dimensional Brownian motion, $\IP_y(\sigma_{a^*}\in ds)$ only depends
on $|y|_\rho$, therefore so does $y\mapsto p^{(X)}(t,x,y)$.
For $x\in \IR_+$ and $y\in D_0$, set $\wt {p}(t,x,r):=p(t,x,y)$ for $r=|y|_\rho$.  For all $a>b>0$ and  $x\in \IR_+$,
\begin{align*}
\int_a^b p^{(Y)}(t,-|x|,y)dy&=  \IP_{-|x|} ( a\leq Y_t\leq b) = \IP_x ( X_t \in D_0 \hbox{ with }  a\leq |X_t |_\rho\leq b) \\
&=  \int_{y\in D_0:  a \leq |y|_\rho \leq b}p (t,x,y)m_p(dy) =\int_{y\in D_0:  a+\eps \leq |y| \leq b+\eps}p (t,x,y)m_p(dy)
\\
&=\int_{a}^b 2\pi (r+\eps)\wt {p}(t,x,r)dr.
\end{align*}
This  implies when $x\in \IR_+$, $y\in D_0$,
\begin{equation}\label{stuff10751}
p^{(Y)}(t,-|x|,|y|_\rho) =2\pi ( |y|_\rho +\eps) \wt {p} (t,x, |y|_\rho ) =
2\pi ( |y|_\rho +\eps)p (t,x,y) .
\end{equation}
We thus have by  Proposition \ref{P:4.4} that
\begin{equation}
\frac{c_1}{\sqrt{t}}e^{-c_2\rho(x,y)^2/t} \le p(t,x,y)\le
\frac{c_3}{\sqrt{t}}e^{-c_4\rho(x,y)^2/t}
\quad \text{ for } x\in \IR_+ \hbox{ and }
 y\in D_0 \text{ with }|y|_\rho<1.
\end{equation}
When $|y|_\rho>1$, we first have
\begin{equation*}
p(t,x,y)=\frac{1}{2\pi (|y|_\rho+\eps)}p^{(Y)}(t,-|x|, |y|_\rho)\lesssim \frac{1}{(|y|_\rho+\eps)\sqrt{t}}e^{-{c_3\rho(x,y)^2}/{t}}
 \leq \frac{1} {\sqrt{t}}e^{-{c_3\rho(x,y)^2}/{t}},
\end{equation*}
while since $\rho(x,y)\ge |y|_\rho>1$,
\begin{align}
p(t,x,y)&=\frac{1}{2\pi (|y|_\rho+\eps)}p^{(Y)}(t,-|x|, |y|_\rho)\gtrsim
\frac{1}{(|y|_\rho+\eps)\sqrt{t}}e^{-{c_4\rho(x,y)^2}/{t}}\nonumber
\\
&\gtrsim  \frac{1}{\sqrt{t }} \frac{\sqrt{t}}{\sqrt{T} \rho (x, y)} 
  e^{-{c_4\rho(x,y)^2}/{t}}
\gtrsim  \frac{1}{\sqrt{t}} e^{-{(c_4+1)\rho(x,y)^2}/{t}}.\label{e:4.8} 
\end{align}
This completes the proof.
\end{proof}

Theorems \ref{T:4.5} and \ref{T:4.6} establish Theorem \ref{T:smalltime}(i).
We next consider part (ii) of Theorem \ref{T:smalltime}
when  both $x$ and $y$ are in
$D_0$.

\begin{thm}\label{T:4.7}
For every $T\geq 1$,
there exist constants $C_i>0$, $13\le i\le 22$, such that for all
$t\in [0, T]$ and $x, y\in D_0$,  the following estimates hold. 
\\
When $\max\{|x|_\rho, |y|_\rho\} \leq 1$,
\begin{align}\label{stuffstuff}
\nonumber &\frac{C_{13}}{\sqrt{t}}e^{-{C_{14}\rho(x,y)^2}/{t}}+\frac{C_{13}}{t}\left(1\wedge
\frac{|x|_\rho}{\sqrt{t}}\right)\left(1\wedge
\frac{|y|_\rho}{\sqrt{t}}\right)e^{-{C_{15}|x-y|^2}/{t}} \le p(t,x,y)
\\
&\le \frac{C_{16}}{\sqrt{t}}e^{-{C_{17}\rho(x,y)^2}/{t}}+\frac{C_{16}}{t}\left(1\wedge
\frac{|x|_\rho}{\sqrt{t}}\right)\left(1\wedge
\frac{|y|_\rho}{\sqrt{t}}\right)e^{-{C_{18}|x-y|^2}/{t}};
\end{align}
and when $\max\{|x|_\rho, |y|_\rho\} >1$,
\begin{equation}\label{stuffstuff*}
\frac{C_{19}}{t}e^{-{C_{20}\rho(x,y)^2}/{t}} \le p(t,x,y) \le \frac{C_{21}}{t}e^{-{C_{22}\rho(x,y)^2}/{t}} .
\end{equation}
Here $|\cdot|$ and $|\cdot|_\rho$ denote the Euclidean metric and
the geodesic metric in $D_0$, respectively.
\end{thm}

\begin{proof}
 For $x\in D_0$ and $t \in (0, T]$, note that
\begin{equation}\label{e:4.10a}
p(t,x,y) =\overline{p}_{D_0}(t,x,y)+p_{D_0}(t,x,y),
\end{equation}
where
\begin{equation}\label{e:4.10}
\overline{p}_{D_0}(t,x,y) =
\int_{0}^t p(t-s, a^*,y) \IP_x (\sigma_{\{a^*\}}\in ds).
\end{equation}

 As  mentioned in the proof for Theorem \ref{T:4.6}, $ p(t-s, a^*,y)$
is  a function in $y$ depending only on $|y|_\rho$.  Therefore so is $y\mapsto
\overline{p}_{D_0}(t,x,y)$. Set $\wt {p}_{D_0}(t,x, r):=\overline{p}_{D_0}(t,x,y)$  for $r=|y|_\rho$. For any $b>a>0$,
\begin{eqnarray*}
\IP_x \left( \sigma_{a^*}<t, \, X_t \in D_0 \hbox{ with } a\leq |X_t|_\rho \leq b\right)
& =& \int_{a\le |y|_\rho\le b}\overline{p}_{D_0} (t,x,y)m_p(dy)
\\
&=& 2 \pi \int_a^b(r+\eps)\wt {p}_{D_0}(t,x,r)dr.
\end{eqnarray*}
On the other hand,
\begin{eqnarray*}
&& \IP_x \left( \sigma_{a^*}<t, \, X_t \in D_0 \hbox{ with } a\leq |X_t|_\rho \leq b\right) \\
&=& \IP_{|x|_\rho}^{(Y)} \left( \sigma_{a^*}<t, \, Y_t>0 \hbox{ with } a\leq |Y_t|_\rho \leq b\right) 
= \int_0^t \left(\int_a^b p^{(Y)}(t-s, 0, r) dr \right) \IP_{|x|_\rho }^{(Y)}
(\sigma_0 \in ds) .
\end{eqnarray*}
It follows that
$$
2\pi (r+\eps) \wt {p}_{D_0}(t,x,r) = \int_0^t p^{(Y)}(t-s, 0, r)
\IP_{|x|_\rho }^{(Y)} (\sigma_{\{0\}} \in ds)
= p^{(Y)} (t, -|x|_\rho , r).
$$
In other words, 
\begin{equation}\label{e:4.11}
\overline p_{D_0} (t, x, y)= \frac{1}{2( | y|_\rho +\eps)} p^{(Y)} (t, -|x|_\rho , |y|_\rho ).
\end{equation}
It is known that the Dirichlet heat kernel $p_{D_0} (t, x, y)$
enjoys the following two-sided estimates:
\begin{eqnarray}
\frac{c_1}{t}\left(1\wedge
\frac{|x|_\rho}{\sqrt{t}}\right)\left(1\wedge
\frac{|y|_\rho}{\sqrt{t}}\right)e^{- {c_2|x-y|^2}/{t}}
\leq p_{D_0} (t, x, y)
\leq \frac{c_3}{t}\left(1\wedge
\frac{|x|_\rho}{\sqrt{t}}\right)\left(1\wedge
\frac{|y|_\rho}{\sqrt{t}}\right)e^{-{c_4|x-y|^2}/t} \nonumber \\
\label{e:4.12}
\end{eqnarray}
for $t\in (0, T]$ and $x, y\in D_0$.
We now consider two different cases.

 \medskip

\noindent{\it Case} (i):  $\max\{|x|_\rho , |y|_\rho \} \leq 1$ and $t\in (0, T]$.   In this case,
it follows from \eqref{e:4.10a}-\eqref{e:4.12}
and Proposition \ref{P:4.4} that
\begin{align}
\nonumber \frac{c_5}{\sqrt{t}}e^{-{c_6(|x|_\rho+|y|_\rho)^2}/{t}}&+\frac{c_5}{t}\left(1\wedge
\frac{|x|_\rho}{\sqrt{t}}\right)\left(1\wedge
\frac{|y|_\rho}{\sqrt{t}}\right)e^{-{c_7|x-y|^2}/{t}}\le p(t,x,y)
\\
\label{7181222} &\le \frac{c_8}{\sqrt{t}}e^{-{c_9(|x|_\rho+|y|_\rho)^2}/{t}}+\frac{c_8}{t}\left(1\wedge
\frac{|x|_\rho}{\sqrt{t}}\right)\left(1\wedge
\frac{|y|_\rho}{\sqrt{t}}\right)e^{-{c_{10}|x-y|^2}/{t}}.
\end{align}
Observe that
\begin{equation}\label{e:4.16}
 (|x|_\rho + |y|_\rho) ^2/t \asymp \rho (x, y)^2/t \quad \hbox{ if  }
  {|x|_\rho} \wedge  {|y|_\rho} \leq {\sqrt{t}} .
\end{equation}
 When $ {|x|_\rho} \wedge  {|y|_\rho} >{\sqrt{t}}$,
 for $a>0$, $b>0$,
 \begin{eqnarray}
 && \frac{1}{\sqrt{t}}e^{- {a (|x|_\rho + |y|_\rho) ^2}/{t}} +\frac{1 }{t}\left(1\wedge
\frac{|x|_\rho}{\sqrt{t}}\right)\left(1\wedge
\frac{|y|_\rho}{\sqrt{t}}\right)e^{- {b |x-y|^2}/{t}}  \nonumber \\
&\asymp & \frac{1}{\sqrt{t}}e^{- {a (|x|_\rho + |y|_\rho) ^2}/{t}} +\left(  \frac1{\sqrt{t}} + \frac{1 }{t} \right)
e^{- {b  |x-y|^2} / {t}} \nonumber  \\
&=& \frac{1}{\sqrt{t}} \left( e^{- {a (|x|_\rho + |y|_\rho) ^2}/{t}}  + e^{- {b  |x-y|^2} / {t}} \right)
+ \frac{1 }{t}    e^{- {b  |x-y|^2} / {t}} \label{e:4.17}
\end{eqnarray}
 The desired estimate \eqref{stuffstuff} now follows from \eqref{7181222}-\eqref{e:4.17} and the
 fact \eqref{e:1.1}.

  \medskip

\noindent{\it Case} (ii):    $\max\{ |x|_\rho, |y|_\rho\} >1$ and $t\in (0, T]$.  By the symmetry of $p(t, x, y)$ in $x$ and $y$,
 in this case we may and do assume $|y|_\rho >1>\sqrt{t/T}$.
 It then follows from \eqref{e:4.11}-\eqref{e:4.12}, Proposition \ref{P:4.4}
 and   \eqref{e:4.8} that
\begin{align}\label{stuff1242}
\nonumber\frac{c_{11}}{t}e^{- {c_{12}(|x|_\rho+|y|_\rho)^2}/{t}}&+\frac{c_{11}}{t}\left(1\wedge
\frac{|x|_\rho}{\sqrt{t}}\right)\left(1\wedge
\frac{|y|_\rho}{\sqrt{t}}\right)e^{- {c_{13}|x-y|^2}/{t}} \le p(t,x,y)
\\
&\le \frac{c_{14}}{t}e^{- {c_{15}(|x|_\rho+|y|_\rho)^2}/{t}}+\frac{c_{14}}{t}\left(1\wedge
\frac{|x|_\rho}{\sqrt{t}}\right)\left(1\wedge
\frac{|y|_\rho}{\sqrt{t}}\right)e^{- {c_{16}|x-y|^2}/{t}} .
\end{align}
 When $ {|x|_\rho} \wedge  {|y|_\rho} \leq {\sqrt{t}}$,
the lower bound estimate \eqref{stuffstuff*} follows from \eqref{stuff1242} and  \eqref{e:4.16},
while the upper bound estimate \eqref{stuffstuff*} follows from Proposition \ref{offdiagUBE}.
Whereas when $ {|x|_\rho} \wedge {|y|_\rho} >{\sqrt{t}}$,
the desired estimate \eqref{stuffstuff*} follows from \eqref{stuff1242}  and \eqref{e:1.1}.
 This completes the proof of the theorem.
\end{proof}

\begin{rmk}\label{rmkonsmalltimeHKE} \rm
\begin{description}
\item{(i)}  One cannot expect to rewrite the estimate of
\eqref{stuffstuff} as $t^{-1}e^{-{c\rho(x,y)^2}/{t}}$.  A
counterexample is that $x=y=a^*$, in which case $x$ and $y$ can be
viewed as either on $\IR$ or on $D_0$,
therefore both Proposition \ref{P:4.4} and Theorem \ref{T:4.6} have already
confirmed that $p(t,x,y)\asymp t^{-1/2}$, which is
consistent with the \eqref{stuffstuff}.

\item{(ii)}  The Euclidean distance appearing in
\eqref{stuffstuff} cannot be replaced with the geodesic distance. To
see this,  take
$x=(\eps+t^{-1/2},0 )$ and $y=(-\eps-t^{-1/2},0 )$ in $D_0$.
The estimate of \eqref{stuffstuff} is
comparable with
$t^{-1/2}+t^{-1}\exp(-{\eps^2}/{t})$, but if
we replaced $|x-y|$ with $\rho(x,y)$, it would be comparable with
$t^{-1/2}+t^{-1}$. For fixed $\eps$, as $t \downarrow 0$,
$t^{-1/2}+t^{-1}\exp(- {\eps^2}/{t})\sim
t^{-1/2}$, but $t^{-1/2}+t^{-1}\sim
t^{-1}$.

\item{(iii)}  Theorem \ref{T:smalltime}  also shows that the parabolic Harnack inequality fails
for $X$.  For a precise statement of the parabolic Harnack inequality, see, for example, \cite{Gr, LSC1, St2}. 
  For $s\in (0, 1]$, take some $y\in D_0$ such that $|y|_\rho=\sqrt{s}$. Set $Q_+:=(3s/2, 2s)\times B_\rho(y, 2\sqrt{s})$ and $Q_-:=(s/2, s)\times B_\rho(y, 2\sqrt{s})$. Let  $u(t,x):=p(t,x,y)$. It follows from Theorem \ref{T:smalltime} 
 that $  \sup_{Q_+}u \asymp s^{-1/2}+s^{-1}\asymp s^{-1}$  and
    $ \inf_{Q_-} u\asymp s^{-1}$.  Clearly there does not exist any positive constant $C>0$ so that  $\sup_{Q_+}u\le C\inf_{Q_-}u$ holds
    for all $s\in (0, 1]$.  This shows that parabolic Harnack inequality fails for $X$. 
\end{description}
\end{rmk}

\section{Large time heat kernel estimates}\label{S:5}

Recall that $X$ denotes the BMVD process on $E$ and its signed radial process
defined by \eqref{710901} is denoted by $Y$.
In this section,  unless otherwise stated, it is always assumed
that $T\geq 8$ and $t\in [T, \infty)$. With loss of generality, we assume
that the radius $\eps$ of the ``hole'' $B(0, \eps)$ satisfies $\eps \leq 1/4$.
 We begin  with the following estimates for the distribution of
hitting time of a disk by a two-dimensional Brownian motion,
which follow directly from  \cite[\S 5.1, Case (a), $\alpha=1$]{GSC}
and a Brownian scaling.

\begin{prop}[Grigor'yan and Saloff-Coste \cite{GSC}]\label{P:5.1}
Let $X$ be a Brownian motion on $\IR^2$ and $K$ be the closed ball with radius $\eps$
centered at the origin.
\begin{description}
\item{\rm (i)} If $0<t<2|x|^2$ and $|x|\geq 1+\eps$, then
\begin{equation}\label{hitting-Bessel2-PDF1}
\frac{c}{\log |x|}\exp\left(-C {|x|^2}/ {t}\right)\le \IP_x (\sigma_K \leq t) \le \frac{C}{\log|x|}\exp\left(-c {|x|^2}/{t}\right),
\end{equation}
for some positive constants $C> c>0$.

\item{\rm (ii)}  If $t\ge 2|x|^2$ and $|x|\geq 1+\eps$, then
\begin{equation}\label{hitting-Bessel2-PDF2}
\IP_x (\sigma_K \leq t) \asymp \frac{\log \sqrt{t}-\log |x|}{\log \sqrt{t}},
\end{equation}
and
\begin{equation}\label{hitting-Bessel2-derivative}
\partial_t \IP_x (\sigma_K \leq t) \asymp \frac{\log |x|}{t(\log t)^2}.
\end{equation}
\end{description}
\end{prop}

Our first goal is to
establish an upper bound estimate on $\int_0^t p(s, a^*, a^*)ds $    the Proposition \ref{1139}.
This will be done through two Propositions by using the above hitting time estimates.

\begin{prop}\label{1125}
$p(t,a^*,a^*)$ is decreasing in $t\in (0, \infty)$.
\end{prop}

\begin{proof} This follows from
\begin{align*}
\frac{d}{dt}p(t,a^*,a^*)&=\frac{d}{dt}\int_x  p(t/2, a^*,x)^2m_p(dx)
\\
&=\int \left(\frac{\partial }{\partial t}p(t/2, a^*,x)\right)p(t/2, a^*,x)m_p(dx)
\\
&=\int \mathcal{L}_x p(t/2,a^*,x)p(t/2,a^*,x)m_p(dx)
\\
&=-\EE\left(p(t/2, a^*,x), p(t/2,a^*,x)\right) \le 0.
\end{align*}
\end{proof}

\begin{prop}\label{1126}
There exists some constant $C_1>0$ such that
\begin{equation*}
p(t,a^*,a^*)\le C_1\frac{\log t}{t} \quad \textrm{for } t\in [8, \infty).
\end{equation*}
\end{prop}

\begin{proof}
For $t\geq 8$ and $x\in D_0$ with $1<|x|_\rho <\sqrt{t/3}$, by Proposition \ref{1125},
\begin{align*}
p(t,x,a^*)&= \int_{0}^t \IP_x(\sigma_{a^*}\in ds)p(t-s, a^*,a^*)
 \ge p(t,a^*,a^*)\IP_x(\sigma_{a^*}\le t)
 \asymp p(t,a^*,a^*)\left(1-\frac{\log |x| }{  \log \sqrt{t}}\right),
\end{align*}
where the $``\asymp "$ is due to \eqref{hitting-Bessel2-derivative}. Therefore,
\begin{align*}
1\ge \IP_{a^*}\left(X_t\in D_0 \hbox{ with } 1<|X_t|_\rho <\sqrt{t/4} \right)
&=\int_{D_0\cap \{1<|x|_\rho <\sqrt{t/4}\}}p(t,a^*,x)dx
\\
&\ge c_1  p(t,a^*,a^*) \int_{1+\eps}^{\sqrt{t/4}+\eps}\left(1-\frac{\log r}{ \log \sqrt{t}}\right)rdr
\\
& \geq c_2 \, p(t,a^*,a^*)  \, \frac{t}{\log t}.
\end{align*}
By selecting $C_2$ large enough, the above yields the desired estimate for $p(t,a^*,a^*) $ for $t\geq 8$.
\end{proof}

\begin{prop}\label{1139}
There exists some $C_2>0$ such that
\begin{equation*}
 \int_0^t p(s,a^*,a^*)ds \le C_2 \log t, \quad \text{ for all }t\geq 4.
\end{equation*}
\end{prop}

\begin{proof}
For  $t\geq 8$ and $x\in  D_0$ with $ 1<|x|_\rho<\sqrt{t/2}$, we have by \eqref{hitting-Bessel2-derivative},
\begin{align*}
p(t, x, a^*)&\ge \int_{t/2}^t p(t-s, a^*,a^*)  \IP_x(\sigma_{a^*}\in ds )
\asymp \frac{\log |x|}{t(\log t)^2}\int_{0}^{t/2}p(s,a^*,a^*)ds.
\end{align*}
Thus by using polar coordinate,
\begin{eqnarray*}
1 & \ge & \IP_{a^*}\left(X_t\in D_0 \hbox{ with } 1<|X_t|_\rho<\sqrt{t/2} \right)
 =\int_{\{x\in D_0: 1<|x|_\rho <\sqrt{t/2} \}} p(t, a^*, x)m_p(dx)
\\
 &\gtrsim& \int_{1}^{\sqrt{t/2}}\frac{\log (r+\eps)}{t(\log t)^2} (r+\eps)dr \cdot \int_0^{t/2}p(s,a^*,a^*)ds
\\
&= &  \frac1{t(\log t)^2} \int_{1+\eps}^{\sqrt{t/2}+\eps} {r\log r} dr \cdot \int_0^{t/2}p(s,a^*,a^*)ds
\\
&\asymp & \frac{t\log t}{t(\log t)^2}\int_0^{t/2}p(s,a^*,a^*)ds
\\
&=& \frac{1}{\log t}\int_0^{t/2}p(s,a^*,a^*)ds.
\end{eqnarray*}
This yields the desired estimate.
\end{proof}

The above proposition  suggests  that $p(t, a^*, a^*)\leq c/t$ for $t\in [4, \infty)$.
However, in order to prove this rigorously, we first compute the upper bounds for $p(t, a^*, x)$ for different regions of $x$,
and then use the identity  $p(t, a^*, a^*)=\int_E p(t/2, a^*, x)^2 m_p(dx)$ to obtain  the sharp upper bound estimate for $p(t, a^*, a^*)$.

\begin{prop}\label{313}
There exists $C_3>0$ such that  for all $t \geq 8$ and  $x\in D_0$ with $1<|x|_\rho<\sqrt{t/2}$,
\begin{equation*}
p(t, a^*, x)\le \frac{C_3}{t}\log \left( \frac{\sqrt{t}}{|x|}\right) .
\end{equation*}

\begin{proof}
By   Proposition \ref{1126}, \eqref{hitting-Bessel2-PDF2} and Proposition \ref{1139},
\begin{align*}
p(t, x, a^*)&= \int_0^{t/2} p(t-s, a^*, a^*) \IP_x(\sigma_{a^*}\in ds) +
\int_{t/2}^t  p(t-s, a^*, a^*) \IP_{x}(\sigma_{a^*}\in ds)
\\
&\le c_1\left( \frac{\log t}{t}\IP_x(\sigma_{a^*}\le t/2)+\frac{\log |x|}{t(\log t)^2}\int_{0}^{t/2}p(s,a^*,a^*)ds\right)
\\
 &\le c_2\left(\frac{\log t }{t}  \, \frac{\log \left( {\sqrt{t}}/{ |x|}\right)}{\log \sqrt{t}}   +\frac{\log |x|}{t\log t}\right)
\\
&\le c_3\left(\frac{1}{t}\log \left( \frac{\sqrt{t} }{|x|} \right)+\frac{1}{t}\right)
\leq \frac{c_4}{t}\log \left( \frac{\sqrt{t} }{|x|}  \right).
\end{align*}
\end{proof}
\end{prop}

The following asymptotic estimate for the distribution of Brownian hitting time of a disk from  \cite[Theorem 2]{KU} will be used in the next proposition.

\begin{lem}[Uchiyama \cite{KU}]\label{KUchiyama}
Let $(B_t)_{t\ge 0}$ be the standard two-dimensional Brownian motion and  $\sigma_r:=\inf\{t>0, |B_t|\le r\}$.
Denote by  $p_{r,x}(t)$ the probability density function of $\sigma_r$ with $B_0=x$.
For every $r_0$, uniformly for $|x|\ge r_0$, as $t\rightarrow \infty$,
\begin{equation*}
p_{r_0, x}(t)=\frac{\log \left(\frac{1}{2} e^{c_0}|x|^2 \right)}{t\left( \log t + c_0 \right)^2}
\exp \left( -\frac{|x|^2}{2t} \right) +
\begin{cases}
O\left(\frac{1+(\log(|x|^2/t))^2}{|x|^2(\log t)^3}\right) \quad
&\hbox{for }\,|x|^2\ge t,  \medskip  \\
\frac{2\gamma \log (t/|x|^2)}{t (\log t)^3}
+O\left(\frac{1 }{t(\log t)^3}\right) \quad
&\hbox{for }\,|x|^2 < t,
\end{cases}
\end{equation*}
where $c_0$ is a positive constant only depending on $r_0$ and
$\gamma = - \int_0^\infty e^{-u} \log u du$ is the Euler constant.
\end{lem}

\begin{prop}\label{157}
There exists $C_4, C_5>0$ such that for all $t\geq 8$ and all $x\in D_0$    ,
\begin{equation*}
p(t, a^*, x)\le \begin{cases}
C_4\left(\frac{1}{t}e^{- {C_5|x|^2}/{t}}+\frac{\left(\log \left( |x|^2/t \right)\right)^2}{|x|^2 \left( \log  t\right)^2}\right)
&\hbox{when } |x|_\rho >\sqrt{t} ,  \\
\frac{C_4}{t} \left( e^{- {C_5|x|^2}/{t}}+\frac{1}{ (\log t)^2} \right)
&\hbox{when } \sqrt{t}/2 \leq |x|_\rho \leq \sqrt{t}.
\end{cases}
\end{equation*}
\end{prop}

\begin{proof} Note that
\begin{eqnarray}
\nonumber p(t, a^*, x)&=&\int_0^t  p(t-s, a^*, a^*) \IP^x(\sigma_{a^*}\in ds)
\\
 \label{1158} &=&\int_{0}^{t/2}p(t-s,a^*,a^*) \IP_x(\sigma_{a^*}\in ds)
 +\int_{t/2}^{t} p(t-s, a^*, a^*) \IP_x(\sigma_{a^*}\in ds).
\end{eqnarray}
By the monotonicity of $p(t,a^*, a^*)$ established in Proposition \ref{1125},
  estimate \eqref{hitting-Bessel2-PDF1} and Proposition \ref{1126},
\begin{eqnarray}
\nonumber \int_{0}^{t/2}p(t-s,a^*,a^*) \IP_x(\sigma_{a^*}\in ds)
&\le & p(t/2,a^*,a^*)\IP_x(\sigma_{a^*}\le t/2)
\\
\nonumber &\stackrel{\eqref{hitting-Bessel2-PDF1}}{\le} & p(t/2,a^*,a^*)\frac{c_1}{\log |x|}e^{-{c_2|x|^2}/{t}}
\\
\nonumber  &\le& \frac{c_1 \log t}{t\log |x| }e^{- {c_2|x|^2}/{t}}
\\
\label{1226}&\le& \frac{c_3}{t}e^{- {c_2|x|^2}/{t}}.
\end{eqnarray}
On the other hand, by Proposition \ref{1139} and Lemma \ref{KUchiyama},
for $|x|_\rho \geq \sqrt{t}$,
\begin{align}
\nonumber \int_{t/2}^t p(t-s, a^*, a^*) \IP^x(\sigma_{a^*}\in ds)
 &\le \sup_{s\in [t/2, t]}p_{\eps, x}(s) \cdot \int_{t/2}^{t}p(t-s, a^*, a^*)ds
\\
\nonumber  &\le c_4  \left(\frac{\log |x|}{t(\log t)^2}e^{- |x|^2/{(2t)}}+\frac{\left(\log \left(|x|^2/t\right)\right)^2}{|x|^2(\log t)^3}\right) \, \log t
\\
\nonumber  &= c_4\left(\frac{\log |x|}{t\log t}e^{- |x|^2/{(2t)}}+\frac{\left(\log \left(|x|^2/t\right)\right)^2}{|x|^2(\log t)^2}\right)
\\
\label{1249}&\leq  \frac{c_4}{t}e^{- c_5|x|^2/t }+c_4 \frac{\left(\log \left(|x|^2/t\right)\right)^2}{|x|^2(\log t)^2},
\end{align}
while for $\sqrt{t}/2\leq |x|_\rho \leq \sqrt{t}$,
\begin{align}
\nonumber \int_{t/2}^t p(t-s, a^*, a^*) \IP_x(\sigma_{a^*}\in ds)
 &\le \sup_{s\in [t/2, t]}p_{\eps, x}(s) \cdot \int_{t/2}^{t}p(t-s, a^*, a^*)ds
\\
\nonumber  &\le c_6  \left(\frac{\log |x|}{t(\log t)^2}e^{- |x|^2/{(2t)}}+\frac{1}{t(\log t)^3}\right) \, \log t
\\
\label{e:5.7}&\leq  \frac{c_6}{t}e^{- c_5|x|^2/t }+  \frac{c_6}{t(\log t)^2}.
\end{align}

The desired estimate now follows from  \eqref{1158}-\eqref{e:5.7}.
\end{proof}

\begin{prop}\label{1253}
There exists $C_6  >0$ such that
\begin{equation*}
p(t, a^*, x)\le  \frac{C_6 \log t}{t}e^{-{ x^2}/{(2t)}}    \qquad \hbox{for all }t\geq 8
\hbox{ and }  x\in \IR_+.
\end{equation*}
\end{prop}

\begin{proof}
Starting from $x\in \IR_+$, BMVD $X$ on $E$ runs like a one-dimensional Brownian motion before hitting $a^*$.
Thus by the known formula for the first passage distribution for one-dimensional Brownian motion,   
\begin{equation}\label{e:5.8}
\IP_x(\sigma_{a^*}  \in dt)=\frac{1}{\sqrt{2\pi t^3}}  x e^{-x^2/(2t)} dt 
\quad \hbox{for } x\in (0, \infty).
\end{equation}
This together with Propositions \ref{1125}--\ref{1139} and a change of variable $s=x^2/r$ gives
\begin{eqnarray*}
\nonumber p(t, a^*, x)&=&
\int_0^{t/2} p(t-s, a^*, a^*) \IP_x(\sigma_{a^*}\in ds) + \int_{t/2}^t p(t-s, a^*, a^*) \IP_x(\sigma_{a^*}\in ds)
\\
\nonumber &  \leq  &  p(t/2, a^*, a^*) \int_0^{t/2}  \frac{1}{\sqrt{2\pi s^3}}  x e^{-x^2/(2s)} ds+\int_{t/2}^t
 p(t-s, a^*, a^*) \frac{1}{\sqrt{2\pi s^3}}  x e^{-x^2/(2s)}  ds
\\  \nonumber
 &\lesssim  & \frac{ \log t}{t} \int_{2x^2/t}^\infty \frac{1}{\sqrt{r}} e^{-r/2} dr
 +\frac{ x}{\sqrt{t^3}}e^{- x^2/t } \int_{t/2}^{t} p(t-s, a^*, a^*)ds
\\  \nonumber
&\lesssim  & \frac{ \log t}{t}   e^{-x^2/t}
 +\frac{1}{t}e^{- x^2/(2t) } \cdot \log t \\
 &\lesssim  & \frac{ \log t}{t}   e^{-x^2/(2t)}.
\end{eqnarray*}
\end{proof}

We are now in a position to establish the following on-diagonal upper bound estimate at $a^*$.

\begin{thm}\label{138}
There exists $C_7>0$ such that
\begin{equation*}
p(t, a^*, a^*)\le {C_7}\left(  t^{-1/2}  \wedge t^{-1} \right)  \quad \hbox{for all } t\in (0, \infty).
\end{equation*}
\end{thm}

\begin{proof}
For  $t\geq 8$,  we have
\begin{align}
p(t, a^*, a^*)&=\int_E p(t/2,a^*,x)^2m_p(dx) \nonumber
\\
&=\left(\int_{\IR_+}+\int_{D_0\cap \{0<|x|_\rho<1\}}+\int_{D_0\cap \{1<|x|_\rho<\sqrt{t/2}\}}
+\int_{D_0\cap \{|x|_\rho>\sqrt{t/2}\}}\right)p(t/2,a^*,x)^2m_p(dx). \label{147}
\end{align}
It follows from Proposition \ref{1253}  that
\begin{equation}\label{144}
\int_{\IR_+}p(t/2,a^*,x)^2m_p(dx)\lesssim \left(\frac{\log t}{t}\right)^2 p
\int_0^\infty e^{- {x^2}/{(2t)}}   dx \asymp \frac{c_1(\log t)^2}{t^{3/2}}.
\end{equation}
By Proposition \ref{offdiagUBE},
 \begin{eqnarray}
\nonumber   \int_{D_0\cap \{0<|x|_\rho<1\}}p(t/2,a^*,x)^2m_p(dx)  &\le& \left(\sup_{x\in D_0: {0<|x|_\rho<1}}
p(t/2,a^*,x)\right)^2m_p(D_0\cap \{0<|x|_\rho<1\})
\\
&\lesssim& \left(\frac{1}{\sqrt{t}}\right)^2= \frac{1}{t}. \label{e:5.11}
\end{eqnarray}
In view of Proposition \ref{157},
\begin{eqnarray}
&& \int_{D_0\cap \{|x|_\rho>\sqrt{t/2}\}}p(t/2,a^*,x)^2m_p(dx)  \nonumber   \\
&  \le & \int_{ D_0\cap \{|x|_\rho>\sqrt{t/2}\}}
\left(\frac{c_1}{t}e^{- {c_2|x|^2}/{t}}\right)^2m_p(dx)
 +\int_{ D_0\cap \{|x|_\rho>\sqrt{t}\}}c_1\left(\frac{\left(\log \left(|x|^2/t \right)\right)^2}{|x|^2 \left( \log  t\right)^2}\right)^2m_p(dx)
 \nonumber \\
 && + \int_{ D_0\cap \{\sqrt{t/2} \leq |x|_\rho\leq \sqrt{t}\}}c_1\left(\frac{1}{t (\log t)^2}  \right)^2m_p(dx) .  \label{616201}
\end{eqnarray}
Using polar coordinate,
\begin{align}\label{206}
 \int_{ D_0\cap \{|x|_\rho>\sqrt{t/2}\}}\left(\frac{c_1}{t}e^{- {c_2|x|^2}/{t}}\right)^2m_p(dx)
 = c_3 \int_{\sqrt{t/2}+\eps }^\infty \frac{r}{t^2}e^{- {c_2r^2}/{t}}dr
 \leq  \frac{c_4}{t},
\end{align}
while
\begin{eqnarray}
&& \int_{ D_0\cap \{|x|_\rho>\sqrt{t}\}} \left(\frac{\left(\log \left( |x|^2/t \right)\right)^2}{|x|^2 \left( \log  t\right)^2}\right)^2m_p(dx)  = 2 \pi  \int_{\sqrt{t}+\eps }^\infty \frac{r\left(\log (r^2/t)\right)^4}{r^4(\log t)^4}dr
\nonumber \\
&\stackrel{u=r/\sqrt{t}}{\leq }& 2\pi\int_{1}^\infty \frac{(\log u)^4}{u^3t^{3/2}(\log t)^4}\,\sqrt{t}\,du
= \frac{2\pi} {t(\log t)^4} \int_{1}^\infty \frac{(\log u)^4}{u^3}\,du = \frac{c_5}{t(\log t)^4},   \label{212}
\end{eqnarray}
and
\begin{equation}\label{e:5.15}
\int_{ D_0\cap \{\sqrt{t/2} \leq |x|_\rho\leq \sqrt{t}\}} \left(\frac{1}{t (\log t)^2}  \right)^2m_p(dx)
\asymp \frac{1}{t (\log t)^4} \lesssim \frac1{t} .
\end{equation}
Hence it follows from \eqref{616201}--\eqref{e:5.15} that
\begin{equation}\label{215}
\int_{ D_0\cap \{|x|_\rho>\sqrt{t/2}\}}p(t/2,a^*,x)^2m_p(dx) \le \frac{c_6}{t}.
\end{equation}
By Proposition \ref{313} and using polar coordinates,
\begin{eqnarray}
&& \int_{D_0\cap \{1<|x|_\rho<\sqrt{t/2}\}}p(t/2,a^*,x)^2m_p(dx)  \nonumber \\
&\lesssim & \frac{1}{t^2}\int_{x\in D_0\cap  \{1<|x|_\rho<\sqrt{t/2}\}}\left(\log \left(\frac{\sqrt{t}}{|x|}\right)\right)^2m_p(dx)  \nonumber \\
 &=&  \frac{2\pi}{t^2}\int_{1+\eps}^{\sqrt{t/2}+\eps}r\left(\log \left(\frac{\sqrt{t}}{r}\right)\right)^2dr   \nonumber \\
 &\stackrel{u=r/\sqrt{t}}{=}& \frac{2\pi }{t} \int_{(1+\eps)/\sqrt{t}}^{(\sqrt{t/2}+\eps)/\sqrt{t}}  u (\log u)^2 du \leq \frac{c_7}{t}.   \label{353}
\end{eqnarray}
Combining \eqref{144},\eqref{e:5.11}, \eqref{215} and \eqref{353}, we conclude that
$p(t, a^*, a^*)\le c_8/t$ for $t\geq 8$.
On the other hand, taking $x=y=0=a^*$ in Theorem \ref{T:4.5} yields that
$ p(t, a^*, a^*)\le c_9 t^{-1/2}$ for $t\in (0, 8]$. This completes the proof of the theorem.
\end{proof}

\begin{thm}\label{T:5.10} There is a constant $C_8\geq 1$ so that
\begin{equation*}
C_8^{-1} \left(t^{-1/2} \wedge t^{-1} \right) \leq p(t, a^*, a^*)
\leq C_8 \left(t^{-1/2} \wedge t^{-1} \right) \quad \hbox{for all } t\in (0, \infty).
\end{equation*}
\end{thm}

\begin{proof} In view of Theorem \ref{138},
it remains to establish  the lower bound estimate.
 By Cauchy-Schwartz inequality,
for $M\geq 1$ to be determined later,
\begin{eqnarray} \label{e:5.18}
p(t, a^*, a^*) &=& \int_E p(t/2, a^*, x)^2 m_p (dx)
\geq \int_{\{x\in E: |x|_\rho \leq M\sqrt{t}\}}
  p(t/2, a^*, x)^2 m_p (dx)
\nonumber \\
&\geq& \frac1{m_p ( \{x\in E: |x|_\rho \leq M\sqrt{t}\})} \left(
\int_{\{x\in E: |x|_\rho \leq M\sqrt{t}\}} p(t/2, a^*, x) m_p (dx)  \right)^2 \nonumber \\
& \gtrsim &\left( t^{-1/2} \wedge t^{-1}\right) \IP_{a^*} (|X_t|_\rho \leq M\sqrt{t})^2.
\end{eqnarray}
We claim that by taking $M$ large enough,
$\IP_{a^*} (|X_t|_\rho \leq M\sqrt{t})\geq 1/2$ for every $t>0$, which will then give the desired
lower bound estimate on $p(t, a^*, a^*)$. Recall the signed radial process $Y=u(X)$ of $X$ from
\eqref{848} satisfies SDE \eqref{YsymmSDE}.
For any $a>0$ and $\delta \in (0, \eps)$, let $Z^{\delta, a}$ and $\wt Z^{\delta, -a}$
be the pathwise unique solution of the following SDEs; see \cite[Theorem 4.3]{BC}:
\begin{eqnarray}
Z^{\delta, a}_t &=& a +B_t + \int_0^t \frac{1}{Z^{\delta, a}_s + \delta} {\bf 1}_{\{ Z^{\delta, a}_s>0\}} ds
+ \wh L^0_t ( Z^{\delta, a}), \label{e:5.19} \\
\wt Z^{\delta, -a}_t &=& -a +B_t +\int_0^t \frac{1}{\wt Z^{\delta, -a}_s -\delta}
{\bf 1}_{\{ \wt Z^{\delta, -a}_s< 0\}} ds
- \wh L^0_t (\wt Z^{\delta, -a}),  \label{e:5.20}
\end{eqnarray}
where $B$ is the Brownian motion in \eqref{YsymmSDE}, and
$\wh L^0 ( Z^{\delta, a})$, $\wh L^0 ( \wt Z^{\delta, -a})$
are the symmetric local times of $ Z^{\delta, a}$ and $  Z^{\delta, -a}$
at 0, respectively.
Denote by $Y^a$ and $Y^{-a}$ the pathwise solutions of
\eqref{YsymmSDE} with $Y^a_0=a$ and $Y^{-a}_0=-a$, respectively.
By comparison principle from \cite[Theorem 4.6]{BC}, we have
with probability one that
$Y^a_t \leq Z^{\delta, a}_t$ for all $t\geq 0$ and
$Y^{-a}_t \geq Z^{\delta, -a}_t$ for all $t\geq 0$.
On the other hand, there are unique solutions to
\begin{eqnarray*}
dZ^a_t = a + B_t+ \int_0^t \frac1{Z^a_s} ds, \qquad
dZ^{-a}_t = -a + B_t+\int_0^t \frac1{Z^{-a}_s} ds .
\end{eqnarray*}
In fact, $Z^a$ and $-Z^{-a}$ are both two-dimensional Bessel processes
on $(0, \infty)$ starting from $a$. They have infinite lifetimes
and never hits $0$. By \cite[Theorem 4.6]{BC} again, diffusion processes
$Z^{\delta, a}$ is decreasing in $\delta$, and $\wt Z^{\delta, -a}$ is increasing in $\delta$.
It is easy to see from the above facts that
$\lim_{\delta \to 0} Z^{\delta, a}_t=Z^a_t$
and $\lim_{\delta \to 0} \wt Z^{\delta, -a}_t= Z^{-a}_t$.
Consequently,  with probability one,
$$
Y^a_t \leq Z^a_t \quad \hbox{and} \quad Y^{-a}_t \geq  Z^{-a}_t
\qquad \hbox{for every } t\geq 0.
$$
In particular, we have for every $t>0$,
$ \IP (Y^a_t \geq M\sqrt{t}) \leq \IP (Z^a_t \geq M\sqrt{t})$ and 
$$
 \IP (Y^{-a}_t \leq -M\sqrt{t}) 
\leq \IP (Z^{-a}_t \leq - M\sqrt{t})= \IP (Z^a_t \geq M\sqrt{t}).
$$
Let $W$ be two-dimensional Brownian motion. Then we have from the above that for every $t>0$,
$$
\IP_a (Y_t \geq M\sqrt{t}) + \IP_{-a} (Y_t \leq -M\sqrt{t})
\leq 2 \IP_{(a, 0)} (|W_t|\geq M\sqrt{t}).
$$
Passing $a\to 0$ yields, by the Brownian scaling property, that 
$$ 
\IP_{a^*}(|X_t|_\rho \geq M\sqrt{t})=
\IP_0 (|Y_t| \geq M\sqrt{t}) \leq 2 \IP_0 (|W_t|\geq M\sqrt{t})
=2 \IP_0 (|W_1|\geq M ),
$$
which is less than 1/2 by choosing $M$ large. This completes the proof of the theorem.
\end{proof}

In the next two propositions, we use the two-sided estimate for $p(t, a^*, a^*)$ as well as Markov property of $X$ 
to get two-sided bounds for $p(t, a^*, x)$ for different regions of $x$.
We first record an elementary lemma that will be used later.

\begin{lem}\label{L:5.11} For every $x>0$,
$$ 
\frac{1}{1+x}\,  e^{-x^2/2} \leq \int_x^\infty e^{-y^2/2} \, dy 
\leq \frac{e\pi}{1+x} e^{-x^2/2}.
$$
\end{lem}

\begin{proof}  Define $\phi (x) = \int_x^\infty e^{-y^2/2} dy -\frac{1}{1+x } e^{-x^2/2}$. Then
$ \phi ' (x)= - \frac{x}{(1+x)^2}  e^{-x^2/2}<0$. Since $\lim_{x\to \infty} \phi (x)=0$, we have 
$\phi (x)>0$ for every $x>0$. This establishes the lower bound estimate of the lemma.
For the upper bound, note that for $x\in (0, 1)$, 
$$ 
\int_x^\infty e^{-y^2/2} \, dy\leq \frac12 \int_{-\infty}^\infty  e^{-y^2/2} \, dy =\sqrt{\frac{\pi}{2}} \leq \sqrt{e \pi /2} \, e^{-x^2/2},
$$
while for every $x>0$, using a change of variable $y=x+z$, 
$$ 
\int_x^\infty e^{-y^2/2} \, dy\leq e^{-x^2/2}  \int_0^\infty  e^{-xz} \, dz =x^{-1} e^{-x^2/2}.
$$
This establishes the upper bound estimate of the lemma. 
\end{proof}

\begin{prop}\label{P:5.12}
There exist constants $C_i>0$, $9\le i\le 10$, so that for all $x\in \IR_+$
and  $t\geq 2$,  
\begin{equation*}
 \frac{C_9}t \left(1+ \frac{|x| \log t }{\sqrt{t}}  \right) 
e^{- {2x^2}/{t}} \le p(t, a^*, x)\le  \frac{C_{10}}t \left(1+ \frac{|x|\log t }{\sqrt{t}}   \right)   e^{- {x^2}/{2t}}.
\end{equation*}
\end{prop}

\begin{proof} Observing that when $x\in \IR_+$,  we have by \eqref{e:5.8},  
\begin{align}
\nonumber 
p(t, a^*, x)&=p(t, x, a^*)=\int_0^t p(t-s, a^*, a^*) 
\IP_x(\sigma_{a^*}\in ds)
\\
 &=\int_0^{t/2} p(t-s, a^*, a^*) \frac{x}{\sqrt{2\pi s^3}}
 e^{-{x^2}/{2s}} ds  +\int_{t/2}^t p(t-s, a^*, a^*)
\frac{x}{\sqrt{2\pi s^3}}e^{-{x^2}/{2s}}ds.
\label{7161042}
\end{align}
By Theorem \ref{T:5.10}, Lemma \ref{L:5.11} and a change of variable
$r=x/\sqrt{s}$,
\begin{eqnarray*}
 \int_0^{t/2} p(t-s, a^*, a^*) \frac{x}{\sqrt{2\pi s^3}}e^{-{x^2}/{2s}} ds 
&\asymp& \frac1t \int_0^{t/2}
 \frac{x}{s^{3/2}}e^{-{x^2}/{2s}} \,ds 
= \frac2{t} \int_{x\sqrt{2/t}}^\infty e^{-r^2/2}\, dr \\
&\asymp& \frac1t \frac1{1 +(x/\sqrt{t})} e^{-x^2/t},
\end{eqnarray*}
while
\begin{eqnarray*}
\frac{x}{\sqrt{2\pi t^3}}e^{-{x^2}/{t}}\int_{0}^{t/2} p(r, a^*, a^*) dr 
 &\le &
\int_{t/2}^{t} p(t-s, a^*, a^*) \frac{x}{\sqrt{2\pi s^3}}e^{-{x^2}/{2s}}\,ds \\ 
&\le& \frac{2x}{\sqrt{\pi t^3}}e^{-{x^2}/{2t}} \int_{0}^{t/2} p(r, a^*, a^*) dr .
\end{eqnarray*}
This, Theorem \ref{T:5.10} and \eqref{7161042} yields the desired result.
\end{proof}

\begin{prop}\label{P:5.13}
There exist $C_i>0$, $11\le i\le 14$ such that
\begin{equation*}
\frac{C_{11}}{t}e^{-{C_{12}|x|_\rho^2}/{t}} \le p(t, a^*, x)\le \frac{C_{13}}{t}e^{-{C_{14}|x|_\rho^2}/{t}} \qquad \text{for all } t \geq 1 \hbox{ and }  x\in D_0.
\end{equation*}
\end{prop}

\begin{proof}
When $t\in [1, 8]$, the estimates follows from  Theorem \ref{T:4.7}. So it remains to establish
the estimates for $t>8$.
We do this by considering three cases.  
 Note that $p(t, a^*, x)=p(t, x, a^*)$.
 
\noindent {\it Case 1. } $1\le|x|_\rho<2\sqrt{t}$. 
We have by Theorem \ref{T:5.10} and Proposition \ref{P:5.1},  
\begin{align*}
p(t, x, a^*)&= \int_0^{t/2}p(t-s, a^*, a^*)\IP_x(\sigma_{a^*}\in ds) 
+\int_{t/2}^t p(t-s, a^*, a^*) \IP_{x}(\sigma_{a^*}\in ds)
\\
&\asymp \frac1{t} \IP_x(\sigma_{a^*}\le t/2)+\frac{\log |x|}{t(\log t)^2}\int_{0}^{t/2}p(s,a^*,a^*)ds
\\
& {\asymp} \frac1t\left(1-\frac{\log |x|}{\log \sqrt{t}}\right)+\frac{\log |x|}{t (\log t)^2}
 \int_0^{t/2} \left( \frac1{\sqrt{s}} \wedge \frac1s\right) ds 
\\
&\asymp \frac{1}{t}\left(1-\frac{\log |x|}{\log \sqrt{t}}\right)+\frac{\log |x|}{t\log t} \asymp \frac{1}{t}.
\end{align*}

\noindent {\it Case 2. } $|x|_\rho\ge2\sqrt{t}$.  In the following computation, $B_\rho(x,r):=\{y\in E: \rho(x, y)< r\}$, and $B_e(x,r) :=\{y\in E: |y-x|<r\}$.
We denote by $\{W; \IP^0_x, x\in \IR^2\}$ 
  two-dimensional Brownian motion and $p^0(t, x, y)=(2\pi t)^{-1} \exp (-|x-y|^2/2t)$ its transition density.
Since
\begin{equation*}
p(t,x,a^*)=\IE_x \left[ p(t-\sigma_{B_\rho(a^*,2)} , X_{\sigma_{B_\rho(a^*,2)}}, a^*);
\sigma_{B_\rho(a^*,2)}<t \right], 
\end{equation*}
it follows from Case 1,  Theorem \ref{T:4.5} and Theorem \ref{T:4.7} that  there is $c_1 \geq 1$ so that 
\begin{eqnarray*}
  p(t,x,a^*) &\lesssim & \IE_x \left[ \left( t-\sigma_{B_\rho(a^*,2)} \right)^{-1} \, e^{-\frac{(2+2\eps)^2}{2c_1(t-\sigma_{B_\rho(a^*,2)})}};
 \sigma_{B_\rho(a^*,2)}<t \right]  \\
 &= & \IE^0_x \left[ \left(t-\sigma_{B_e(0,2+\eps)} \right)^{-1} \, e^{-\frac{ (2+2\eps )^2 }{2c_1 (t-\sigma_{B_e(0,2+\eps)})}};
 \sigma_{B_e(0,2+\eps)}<t \right] \\
 &\lesssim & \IE^0_x \left[ \left(c_1t-\sigma_{B_e(0,2+\eps)} \right)^{-1} \, e^{-\frac{ (2+ \eps )^2 }{2 (c_1 t-\sigma_{B_e(0,2+\eps)})}};
 \sigma_{B_e(0,2+\eps)}<t \right]  \\
 &\leq   & \IE^0_x \left[ \left(c_1 t-\sigma_{B_e(0,2+\eps)} \right)^{-1} \, e^{-\frac{ (2+\eps )^2 }{2 (c_1 t-\sigma_{B_e(0,2+\eps)})}};
 \sigma_{B_e(0,2+\eps)}< c_1 t \right]  \\
 &\asymp & \IE^0_x \left[  p^0(c_1  t-\sigma_{B_e(0,2+\eps)}), W_{\sigma_{B_e(0,2+\eps)}}, 0);
 \sigma_{B_e(0,2+\eps)}< c_1 t \right] \\
 &=& p^0(c_1 t, x, 0) =(2\pi c_1 t)^{-1} \exp (-|x|^2/2c_1t). 
\end{eqnarray*}
 Similarly, for the lower bound estimate,  
  it follows from Case 1, Theorem \ref{T:4.5} and Theorem \ref{T:4.7} that  there is $c_2 \in (0, 1]$ so that 
  \begin{eqnarray*}
  p(t,x,a^*) &\gtrsim  & \IE_x \left[ \left( t-\sigma_{B_\rho(a^*,2)} \right)^{-1}\, e^{-\frac{2^2}{2c_2(t-\sigma_{B_\rho(a^*,2)})}};
 \sigma_{B_\rho(a^*,2)}<t \right]  \\
 &\geq  & \IE^0_x \left[ \left(t-\sigma_{B_e (0,2+\eps)} \right)^{-1} \, e^{-\frac{ 2^2 }{2c_2 (t-\sigma_{B_e (0,2+\eps)})}};
 \sigma_{B_e(0,2+\eps)}<c_2 t \right] \\
 & \gtrsim & \IE^0_x \left[ \left(c_2t-\sigma_{B_e(0,2+\eps)} \right)^{-1} \, e^{-\frac{ (2+\eps )^2 }{2  (c_2 t-\sigma_{B_e(0,2+\eps)})}};
 \sigma_{B_e(0,2+\eps)}<c_2 t \right]  \\
  &\asymp & \IE^0_x \left[  p^0(c_2  t-\sigma_{B_e(0,2+\eps)}), W_{\sigma_{B_e(0,2+\eps)}}, 0);
 \sigma_{B_e(0,2+\eps)}< c_2 t \right] \\
 &=& p^0(c_2 t, x, 0) =(2\pi c_2 t)^{-1} \exp (-|x|^2/2c_2t). 
\end{eqnarray*}
  Realizing that $|x|_\rho>2\sqrt{t}>4\sqrt{2}$ implies that $|x|_\rho\asymp |x|$, we get the desired estimates
  in this case.
  
\noindent  {\it Case 3. } $0<|x|_\rho<1$. Note that by Proposition \ref{P:3.2}, 
\begin{equation}\label{521408}
\int_{y\in D_0\cap B_\rho(a^*, 2)}p(t/2, a^*,y)p(t/2, y, x)m_p(dy) \le \int_{D_0\cap B_\rho(a^*, 2)}\left(\frac{c_1}{\sqrt{t}}\right)^2m_p(dy) 
\asymp \frac{1}{t}, 
\end{equation}
while by Cases 1 and 2 above, 
\begin{equation}\label{521409}
  \int_{D_0\cap B_\rho(a^*, 2)^c}p(t/2, x,y)p(t/2, a^*, y)m_p(dy)  
 \le \sup_{y\in D_0\cap B_\rho(a^*, 2)^c}p(t/2,a^*,y) 
\lesssim \frac1t .   
\end{equation}
On the other hand, by Theorem \ref{T:5.10} again, 
\begin{align}
\nonumber \int_{\IR_+}p(t/2, a^*, y)p(t/2, y,x)m_p(dy)
&=\IE_x \left[ p(t/2, X_{t/2}, a^*); X_{t/2} \in \IR_+ \right]
\\
\nonumber &=\IE_x \left[ p(t/2, X_{t/2}, a^*);  \sigma_{a^*} <t/2 \hbox{ and } X_{t/2} \in \IR_+ \right]
\\
\nonumber &=\IE_x \left[ \IE_{a^*} [p(t/2, X_{t/2-s}, a^*) ] |_{s=\sigma_{a^*}};  \sigma_{a^*} <t/2 \hbox{ and } X_{t/2} \in \IR_+ \right]
\\
\nonumber & = \IE_x \left[   p(t-\sigma_{a^*}, a^*, a^*) ;  \sigma_{a^*} <t/2
\hbox{ and }  X_{t/2}\in \IR_+  \right]
\\
&\asymp \frac{1}{t}\,  \IP_x\left(\sigma_{a^*}\le t/2 
\hbox{ and } X_{t/2}\in \IR_+ \right) \leq  \frac{1}{t}.   \label{521410} 
\end{align}
The estimates  \eqref{521408}, \eqref{521409} and \eqref{521410} imply that 
\begin{align*}
\nonumber p(t,a^*, x)  = & \int_{D_0\cap B_\rho(a^*, 2)}p(t/2, a^*, y)p(t/2, y,x)m_p(dy)
 +\int_{D_0\cap B_\rho^c(a^*, 2)}p(t/2, a^*, y)p(t/2, y,x)m_p(dy)   \\
 &+\int_{\IR_+}p(t/2, a^*, y)p(t/2, y,x)m_p(dy) 
\lesssim \frac1t.
 \end{align*}
On the other hand, there is a constant $c_3>0$ so that $\IP_x (\sigma_{a^*}\leq 1)\geq c_3$
for all $x\in D_0$ with $|x|_\rho \leq 1$. Hence we have by Theorem \ref{T:5.10} that 
for $t\geq 2$ and $x\in D_0$ with $|x|_\rho \leq 1$, 
$$ 
p(t, a^*, x)=p(t, x, a^*)\geq \int_0^1 p(t-s, a^*, a^*)\IP_x (\sigma_{a^*}\in ds)
\gtrsim \frac1t \IP_x (\sigma_{a^*}\leq 1) \gtrsim \frac1t.
$$
In conclusion, we have $p(t,  a^*, x)\asymp \frac1t$ for $t\geq 4$ and $x\in D_0$ with $|x|_\rho \leq 1$.
This completes the proof of the proposition. 
\end{proof}

We are now in the position to derive estimates on $p(t,x,y)$ for   $(x,y)\in D_0\times D$ and  $t\in [8, \infty)$, by 
using the two-sided estimate of $p(t, x, a^*)$ and the Markov property of $X$.  

\begin{thm} \label{154}
There exist constants $C_i>0$, $15 \le i \le 18$, such that the following estimate holds:
\begin{equation*}
\frac{C_{15}}{t}e^{- {C_{16}\rho(x,y)^2}/{t}} \le p(t,x,y)\le \frac{C_{17}}{t}e^{- {C_{18}\rho(x,y)^2}/{t}}, 
\quad (t,x,y)\in [8, \infty)\times D_0\times D_0.
\end{equation*}
\end{thm}

\begin{proof}  
As before,  denote by $\{W; \IP^0_x, x\in \IR^2\}$ 
  two-dimensional Brownian motion and $p^0(t, x, y)=(2\pi t)^{-1} \exp (-|x-y|^2 /2t)$ its transition density.
We first note that, 
as a special case of \cite[Theorem 1.1(a)]{Z3}, 
there are constants $c_1>c_2>0$ so that 
for $t\geq 1$ and $x, y\in D_0$, 
$$ 
\left( |x|_\rho \wedge 1 \right) \left(|y|_\rho \wedge 1 \right) t^{-1} e^{-c_1 |x-y|^2 /t}
\lesssim p^0_{D_0} (t, x, y) \lesssim 
\left( |x|_\rho \wedge 1 \right) \left(|y|_\rho \wedge 1 \right)
 t^{-1} e^{-c_2 |x-y|^2 /t} .
$$
It follows that there is $c_3\in (0, 1]$ so that 
\begin{equation}\label{e:5.25} 
 p^0_{D_0} (t, x, y) \gtrsim p^0_{D_0} (c_3t, x, y) \quad \hbox{for every } t\geq 1 \hbox{ and } x, y\in D_0. 
\end{equation}

We will prove the theorem by considering two different cases.
\\
{\it Case 1. } $|x|_\rho+|y|_\rho >2$.  
Without loss of generality, we assume $|y|_\rho>1$. 
In this case, it is not hard to verify that
$\rho(x,y)\asymp |x-y|$.
 Recall from \eqref{e:1.6}-\eqref{e:1.7} that  for $x, y\in D_0$, 
 $$
 \overline{p}_{D_0}(t,x,y):=p(t, x, y)- p_{D_0} (t, x, y)= \IE_x \left[ p(t-\sigma_{a^*}, a^*, y); \sigma_{a^*}<t \right].
 $$ 
By    Proposition \ref{P:5.13} and the assumption that   $\eps \in (0, 1/4]$, there are constants $c_4\geq 1$ so that
for every $x, y\in D_0$ with $|y|_\rho >1$
\begin{eqnarray}
\overline{p}_{D_0}(t,x,y) 
&\lesssim& \int_0^t  \frac{1}{t-s}e^{-\frac{ (|y|_\rho +3\eps)^2}{2c_4(t-s)}} \,  \IP_x\left(\sigma_{a^*}\in ds\right) 
 \nonumber \\
&\leq & \int_0^{c_1t}  \frac{1}{c_4t-s}e^{-\frac{ (|y|_\rho +2\eps)^2}{2(c_4t-s)}} \,  \IP^0_x\left(\sigma_{B_e(0, \eps)}\in ds\right) 
 \nonumber \\
&\lesssim & \IE_x^0 \left[ p^0(c_4t-s, W_{\sigma_{B_e(0, \eps)}}, y); \sigma_{B_e(0, \eps)}<c_4 t \right] 
 \nonumber \\
&\leq & p^0 (c_4 t, x, y)  .   \label{e:5.26}
\end{eqnarray}
Similarly, there is a constant $c_5 \in (0, c_3]$ so that 
\begin{eqnarray}
\overline{p}_{D_0}(t,x,y) 
&\gtrsim & \int_0^t  \frac{1}{t-s}e^{-\frac{  (|y|_\rho -\eps)^2}{2c_5(t-s)}} \,  \IP_x\left(\sigma_{a^*}\in ds\right)   \nonumber \\
&\geq & \int_0^{c_5t}  \frac{1}{c_5t-s}e^{-\frac{  |y|_\rho^2}{2(c_5t-s)}} \,  \IP^0_x\left(\sigma_{B_e(0, \eps)}\in ds\right)  \nonumber \\
&\gtrsim & \IE_x^0 \left[ p^0(c_5t-s, W_{\sigma_{B_e(0, \eps)}}, y); \sigma_{B_e(0, \eps)}<c_5t \right]  \nonumber \\
&= & \overline p^0_{D_0}  (c_5t, x, y) .   \label{e:5.27}
\end{eqnarray}
Since $p_{D_0}(t, x, y)=p^0_{D_0}(t, x, y)$ and $c_4\geq 1$, we have from \eqref{e:5.26} that 
$$ p(t, x, y)= p_{D_0}(t, x, y)+ \overline p_{D_0}(t, x, y)  \lesssim p^0(t, x, y)+p^0(c_4t, x, y)\lesssim p^0(c_4t, x, y)
\lesssim t^{-1} e^{-\rho (x, y)^2/2c_4t} .
$$
On the other hand, we have by \eqref{e:5.25} and \eqref{e:5.27} that  for every $t\geq 1$ and $x, y\in D_0$ satisfying $|x|_\rho+|y|_\rho>2$, 
$$ 
p(t, x, y)  
\gtrsim p^0_{D_0} (c_5 t, x, y) + \overline p^0_{D_0}(c_5 t, x, y)
=p^0(c_5t, x, y) \gtrsim t^{-1}  e^{-|x-y|^2/2c_5t}  \gtrsim t^{-1}  e^{-\rho (x, y)^2/c_5t} ,
$$
where the last ``$\gtrsim$" is due to the fact that $|x-y|/\sqrt{2}\le \rho(x,y)$, which can be verified easily from the assumptions that $|x|_\rho+|y|_\rho>2$ and $\eps\le 1/4$. This establishes the desired two-sided estimates in this case. 
 
\noindent {\it Case 2. }  $|x|_\rho+|y_\rho \le 2$. 
In this case, it suffices to show that $p(t, x, y)\asymp t^{-1}$ for $t\geq 8$. 
The proof  is similar to that of Case $3$ of Proposition \ref{P:5.13}.  
By Proposition \ref{P:3.2},  for $t\geq 8$, 
\begin{equation}\label{e:5.28}
\int_{y\in D_0\cap B_\rho(a^*, 2)}p(t/2, x,z)p(t/2, z, y)m_p(dy) \le \int_{D_0\cap B_\rho(a^*, 2)}\left(\frac{c_6}{\sqrt{t}}\right)^2m_p(dy) 
\asymp \frac{1}{t}, 
\end{equation}
while by Case 1,
\begin{equation} \label{e:5.29}
  \int_{D_0\cap B(a^*, 2)^c}p(t/2, x,z)p(t/2, z, y)m_p(dz)  \le \sup_{z\in D_0\cap B(a^*, 2)^c}p(t/2,  y,z) 
\lesssim \frac1t. 
\end{equation}
On the other hand, by Proposition \ref{P:5.13},  for $t\geq 8$, 
\begin{align}
\nonumber \int_{\IR_+}p(t/2, x, z)p(t/2, z, y)m_p(dz)
&=\IE_x \left[ p(t/2, X_{t/2}, y); X_{t/2} \in \IR_+ \right]
\\
\nonumber &=\IE_x \left[ p(t/2, X_{t/2}, y);  \sigma_{a^*} <t/2 \hbox{ and } X_{t/2} \in \IR_+ \right]
\\
\nonumber &=\IE_x \left[ \IE_{a^*} [p(t/2, X_{t/2-s}, y) ] |_{s=\sigma_{a^*}};  \sigma_{a^*} <t/2 \hbox{ and } X_{t/2} \in \IR_+ \right]
\\
\nonumber & = \IE_x \left[   p(t-\sigma_{a^*}, a^*, y) ;  \sigma_{a^*} <t/2
\hbox{ and }  X_{t/2}\in \IR_+  \right]
\\
&\asymp \frac{1}{t}\, \IP_x\left(\sigma_{a^*}\le t/2 
\hbox{ and } X_{t/2}\in \IR_+ \right) \leq  \frac{1}{t}.   \label{e:5.30} 
\end{align}
The estimates  \eqref{e:5.28}-\eqref{e:5.30}   imply that 
\begin{align*}
\nonumber p(t,a^*, x)  = & \int_{D_0\cap B_\rho(a^*, 2)}p(t/2, a^*, y)p(t/2, y,x)m_p(dy)
 +\int_{D_0\cap B_\rho^c(a^*, 2)}p(t/2, a^*, y)p(t/2, y,x)m_p(dy)   \\
 &+\int_{\IR_+}p(t/2, a^*, y)p(t/2, y,x)m_p(dy) \\
\lesssim & \frac1t.
 \end{align*}
On the other hand, there is a constant $c_7>0$ so that $\IP_x (\sigma_{a^*}\leq 1)\geq c_7$
for all $x\in D_0$ with $|x|_\rho \leq 2$. Hence we have by Proposition \ref{P:5.13} that 
for $t\geq 2$ and $x\in D_0$ with $|x|_\rho \leq 1$, 
$$ 
p(t, x, y) \geq \int_0^1 p(t-s, a^*, y)\IP_x (\sigma_{a^*}\in ds)
\gtrsim \frac1t \IP_x (\sigma_{a^*}\leq 1) \gtrsim \frac1t.
$$
Therefore we have $p(t,  x, y)\asymp \frac1t$ for $t\geq 8$ and $x, y\in D_0$ with $|x|_\rho +|y|_\rho \leq 2$.
This completes the proof of the theorem.   
\end{proof}

 \begin{thm}\label{T:5.15}
There exist constants $C_i>0$, $19 \le  i \le 26$, such that the following estimates hold for $(t,x,y)\in [4, \infty)\times \IR_+\times D_0$:
when $|y|_\rho<1$,
\begin{equation}\label{1233}
\frac{C_{19}  }{t} \left( 1 + \frac{|x| \log t}{\sqrt{t}} \right)
e^{- {C_{20}\rho(x,y)^2}/{t}} \le p(t,x,y)\le \frac{C_{21}  }{t} \left( 1 + \frac{|x| \log t}{\sqrt{t}} \right)    
e^{- {C_{22}\rho(x,y)^2}/{t}};
\end{equation}
while for  $|y|_\rho\ge 1$,
\begin{eqnarray}
  &&  \frac{C_{23}}{t}
 \left( 1 + \frac{|x|}{\sqrt{t}}  \log \left( 1+\frac{\sqrt{t}}{|y|_\rho} \right)   \right)
 e^{- {C_{24}\rho(x,y)^2}/{t}}  \nonumber \\
&\le & p(t,x,y)    \le  \frac{C_{25}}{t}\left( 1 + \frac{|x|}{\sqrt{t}}  \log \left( 1+\frac{\sqrt{t}}{|y|_\rho} \right)   \right) 
e^{- {C_{26}\rho(x,y)^2}/{t}}.
\label{e:5.32}
\end{eqnarray}
\end{thm}

\begin{proof}
First, note that by Proposition \ref{P:5.13}
and Theorem \ref{T:4.7}, 
\begin{equation}\label{e:5.33}
 \frac{1}{t}e^{-{c_1|y|_\rho^2}/{t}}\lesssim p(t,x,y)\lesssim \frac{1}{t}e^{-{c_2|y|_\rho^2}/{t}}  \quad \hbox{for }
 t \geq 1 \hbox{ and } y\in D_0 
  \end{equation}
and 
\begin{equation}\label{e:5.34}
\frac{1}{\sqrt{t}}e^{- {c_3|y|_\rho^2}/{t}} \lesssim  p(t,a^*,y)\lesssim \frac{1}{\sqrt{t}}e^{- {c_4|y|_\rho^2}/{t}}  \quad 
\hbox{for } t\leq 1 \hbox{ and } y\in D_0 \hbox{ with } |y|_\rho \leq 1.
\end{equation}
By \eqref{e:5.8}, 
\begin{equation}\label{e:5.35}
p(t,x,y)  =  \int_0^tp(t-s,a^*,y) \IP_x(\sigma_{a^*}\in ds)
=\int_0^t p(t-s,a^*,y) \frac{|x|}{\sqrt{2\pi s^3}}    e^{-x^2/(2s)}  ds   .
\end{equation}
It follows from \eqref{e:5.33} and Lemma \ref{L:5.11}  that for every $y\in D_0$ and $t\geq 4$, 
\begin{eqnarray}
&& \int_0^{t/2} p(t-s,a^*,y) \frac{ |x|}{\sqrt{2\pi s^3}}    e^{-x^2/(2s)}  ds
 \lesssim    -  \frac{1}{t}e^{-{c_{2}|y|_\rho^2}/{t}}\int_{s=0}^{t/2}e^{-\frac{ |x|^2}{2s}}d\left(\frac{|x|}{\sqrt{s}}\right) 
 \nonumber \\
&\lesssim &   \frac1t e^{- c_2 |y|_\rho^2 /t} \frac{1}{1+|x|/\sqrt{t}} e^{-|x|^2/t} \leq \frac 1t 
  e^{-c_3 \rho (x, y)^2/t}.       \label{e:5.36}
\end{eqnarray} 
Similarly, we have
\begin{equation}\label{e:5.37} 
\int_0^{t/2} p(t-s,a^*,y) \frac{x}{\sqrt{2\pi s^3}}    e^{-x^2/(2s)}  ds
\gtrsim  \frac1t    e^{-(|x|^2+c_4 |y|_\rho^2)/t} 
\gtrsim \frac1t  e^{-c_5 \rho (x, y)^2/t}.
\end{equation}
We now consider two cases depending on the range of the values of  $|y|_\rho$.

\noindent {\it Case 1. } $y \in D_0$ with $|y|_\rho<1$.  
In this case, we have by \eqref{e:5.33}
\begin{eqnarray*}
 && \int_{t/2}^{t-1} p(t-s,a^*,y) \frac{x}{\sqrt{2\pi s^3}}    e^{-x^2/(2s)}  ds
\lesssim \int_{t/2}^{t-1} \frac{1}{t-s}e^{-{c_2 |y|_\rho^2}/{(t-s)}}  \frac{|x|}{\sqrt{ s^3}}    e^{-x^2/(2s)}  ds \\
&\lesssim & \frac{|x|}{t^{3/2}} e^{-|x|^2 /2t} \int_{t/2}^{t-1}  \frac{1}{t-s} d s 
\lesssim \frac{|x| \log t }{t^{3/2}} e^{-|x|^2 /(2t)} .
  \end{eqnarray*}
Similarly, we have
$$
\int_{t/2}^{t-1} p(t-s,a^*,y) \frac{x}{\sqrt{2\pi s^3}}    e^{-x^2/(2s)}  ds
\gtrsim \frac{|x| \log t }{t^{3/2}} e^{- |x|^2 /t}. 
$$
On the other hand,  by \eqref{e:5.34}, 
\begin{eqnarray*}
 && \int_{t-1}^t p(t-s,a^*,y) \frac{x}{\sqrt{2\pi s^3}}    e^{-x^2/(2s)}  ds
\lesssim \frac{|x|}{\sqrt{ t^3}}    e^{-x^2/(2t)} \int_{t-1}^t \frac{1}{\sqrt{t-s}}e^{-{c_4 |y|_\rho^2}/{(t-s)}}    ds \\
& \lesssim & \frac{|x|}{\sqrt{ t^3}}    e^{-x^2/(2t)} \int_{t-1}^t \frac{1}{\sqrt{t-s}}     ds 
\lesssim \frac{|x|}{\sqrt{ t^3}}    e^{-x^2/(2t)}  ,
  \end{eqnarray*}
  and similarly
  $$ 
  \int_{t-1}^t p(t-s,a^*,y) \frac{x}{\sqrt{2\pi s^3}}    e^{-x^2/(2s)}  ds
  \gtrsim \frac{|x|}{\sqrt{ t^3}}    e^{-x^2/t}  .
$$
These estimates together with \eqref{e:5.35}-\eqref{e:5.37} establishes \eqref{1233}. 

\medskip

\noindent {\it Case 2. } $y\in D_0$ with $|y|_\rho>1$.  Note that by \eqref{e:5.8}, 
\begin{align}
p(t,x,y)&=\int_0^t p(t-s,a^*,y) \IP_x(\sigma_{a^*}\in ds)\label{e:5.38}
\\
  &=\int_0^{t/2} p(t-s,a^*,y) \frac{|x|}{\sqrt{2\pi s^3}}e^{-{x^2}/{2s}} ds
  + \int_{t/2}^{t}  p(t-s,a^*,y) \frac{|x|}{\sqrt{2\pi s^3}}e^{-{x^2}/{2s}} ds. \nonumber
\end{align}
 By Theorem \ref{T:4.7}, \eqref{e:5.33} and a change of variable $r=|y|_\rho /\sqrt{t-s}$, we have
 \begin{eqnarray*}
   \int_{t/2}^{t}  p(t-s,a^*,y) \frac{|x|}{\sqrt{2\pi s^3}}e^{-{x^2}/{2s}} ds
 &\lesssim &\int_{t/2}^{t}  \frac{1}{t-s}e^{-{c_6|y|_\rho^2}/{(t-s)}} \frac{|x|}{\sqrt{s^3}}e^{-{x^2}/{2s}} ds \\
 &\lesssim & \frac{|x|}{ t^{3/2}}e^{-{x^2}/{2t}}\int_{t/2}^{t}  \frac{1}{t-s}e^{-{c_6|y|_\rho^2}/{(t-s)}}  ds  \\ 
 &= & \frac{|x|}{ t^{3/2}}e^{-{x^2}/{2t}}  \int_{ |y|_\rho/ \sqrt{2/t}}^\infty \frac2{r} e^{-c_6 r^2} dr .
 \end{eqnarray*}
 Note that for  each fixed $a>0$,
 $\int_\lambda^\infty r^{-1} e^{-a r^2} dr =  \int_\lambda^1 r^{-1} e^{-a r^2} dr + \int_1^\infty r^{-1} e^{-c_2 r^2} dr$
 is comparable to $  \log (1/\lambda )$ when $0<\lambda \leq 1/2$. For $\lambda \geq 1/2$,  by Lemma \ref{L:5.11}
 $$ \int_\lambda^\infty r^{-1} e^{-a r^2} dr \leq 2 \int_\lambda ^\infty   e^{-a r^2} dr
 =\frac{2}{\sqrt{2a}} \int_{\sqrt{2a}\lambda}^\infty e^{-s^2/2} ds \lesssim \frac1{1+\sqrt{a}\lambda} e^{-a \lambda^2} 
 \leq e^{-a \lambda^2}
 $$
 and 
 $$ 
 \int_\lambda^\infty r^{-1} e^{-a r^2} dr 
 \gtrsim  \int_\lambda ^\infty   e^{-2 a r^2} dr
 =\frac{1}{2\sqrt{a}} \int_{2\sqrt{ a}\lambda}^\infty e^{-s^2/2} ds \gtrsim \frac1{1+\sqrt{a}\lambda} e^{-2a \lambda^2} 
 \gtrsim   e^{-3a \lambda^2}
 $$
Hence we have 
\begin{equation} 
\log (1+\lambda^{-1}) e^{-3a \lambda^2} \lesssim \int_\lambda^\infty r^{-1} e^{-ar^2} dr \leq \log (1+\lambda^{-1}) e^{- a \lambda^2}
\quad \hbox{for any } \lambda >0.
\end{equation}
Thus we have
 \begin{eqnarray*}
   \int_{t/2}^{t}  p(t-s,a^*,y) \frac{|x|}{\sqrt{2\pi s^3}}e^{-{x^2}/{2s}} ds
 &\lesssim & \frac{|x|}{ t^{3/2}}e^{-{x^2}/{2t}} \log \left(1+\frac{\sqrt{t}}{|y|_\rho} \right) e^{-2c_6 |y|_\rho^2/t} \\
 &\leq & \frac{|x|}{ t^{3/2}} \log \left(1+\frac{\sqrt{t}}{|y|_\rho} \right) e^{-c_7\rho (x, y)^2/t}.
 \end{eqnarray*}
 Similarly, we have
 $$ 
\int_{t/2}^{t}  p(t-s,a^*,y) \frac{|x|}{\sqrt{2\pi s^3}}e^{-{x^2}/{2s}} ds
\gtrsim \frac{|x|}{ t^{3/2}} \log \left(1+\frac{\sqrt{t}}{|y|_\rho} \right) e^{-c_8 \rho (x, y)^2/t}.
$$
These together with \eqref{e:5.36}-\eqref{e:5.37} establishes \eqref{e:5.32}.  
\end{proof}

We will need the following elementary lemma.

\begin{lem}\label{L:5.16}
For every $c>0$, there exists $C_{27} \geq 1$ such that for every  $t\geq 3$ and  $0<y\leq \sqrt{t}$,  
\begin{equation*}
  C_{27}^{-1} \log t \leq \int_{2}^t  \frac1s \left( 1+ \frac{ y\log s}{\sqrt{s} }\right)  e^{- {c|y|^2}/{s}}ds\le C_{27}\log t.
\end{equation*}
\end{lem}

\begin{proof}
 By a change of variable $r=y/\sqrt{s}$, 
$$
 \int_{2}^t  \frac{ y\log s}{s^{3/2} }  e^{-{c|y|^2}/{s}}ds
 \le  \log t\int_{2}^t\frac{y}{ s^{3/2}}e^{-{c|y|^2}/{s}}ds
 = 2 \log t\int_{ {y}/{\sqrt{t}}}^{y/\sqrt{2}}e^{-cr^2}dr \lesssim  \log t,
$$
while since $0<y\leq \sqrt{t}$, 
$$ \int_{2}^t  \frac1s  e^{-{c|y|^2}/{s}}ds  \asymp \int_{2}^t \frac1s   ds \asymp \log t  .
$$
 This proves the lemma. 
\end{proof}

\begin{thm}\label{T:5.17}
 There exist constants $C_i>0$, $28\le i\le 34$, such that the following estimate holds for all $(t,x,y)\in [8, \infty)\times \IR_+\times \IR_+$:
\begin{align}\label{1241}
\nonumber &\frac{C_{28}}{\sqrt{t}}\left(1\wedge \frac{|x|}{\sqrt{t}}\right)\left(1\wedge \frac{|y|}{\sqrt{t}}\right)
e^{- {C_{29}|x-y|^2}/{t}}+\frac{C_{28}}{t} \left( 1+ \frac{(|x|+|y|)\log t}{\sqrt{t}}\right)e^{- {C_{30}(x^2+y^2)}/{t}} \le p(t,x,y)
\\
&\le \frac{C_{31}}{\sqrt{t}}\left(1\wedge \frac{|x|}{\sqrt{t}}\right)\left(1\wedge \frac{|y|}{\sqrt{t}}\right)e^{- {C_{32}|x-y|^2}/{t}}+\frac{C_{31}}{t} \left( 1+ \frac{(|x|+|y|)\log t}{\sqrt{t}}\right) e^{-{C_{33}(x^2+y^2)}/{t}}.
\end{align}
\end{thm}

\begin{proof} When either $x=a^*$ or $y=a^*$, this has been established in Proposition \ref{P:5.12}
so we assume $|x|\wedge |y|>0 $.
For simplicity, denote $\IR_+\setminus \{a^*\}$ by $(0, \infty)$.
Since
$$ 
p_{(0, \infty)}(t,x,y)= (2\pi t)^{-1/2} \left( e^{-|x-y|^2/2t} - e^{-|x+y|^2/2t}\right)
=(2\pi t)^{-1/2}   e^{-|x-y|^2/2t}  \left( 1- e^{-2 xy/t}\right),
$$
there are constants $c_1>c_2>0$ so that 
\begin{equation}\label{e:5.41} 
\frac{1}{\sqrt{t}}\left(1\wedge \frac{|x|}{\sqrt{t}}\right)\left(1\wedge \frac{|y|}{\sqrt{t}}\right)e^{- {c_1 |x-y|^2}/{t}}
\lesssim p_{(0, \infty)}(t,x,y) \lesssim 
\frac{1}{\sqrt{t}}\left(1\wedge \frac{|x|}{\sqrt{t}}\right)\left(1\wedge \frac{|y|}{\sqrt{t}}\right) e^{- {c_2 |x-y|^2}/{t}} .
\end{equation}
for all $t>0$ and $x, y\in \IR_+$.  Note also
\begin{equation}\label{e:5.42}
p(t,x,y)=p_{(0, \infty)}(t,x,y)+\int_0^t p(t-s,a^*, y) \IP_x(\sigma_{a^*}\in ds).
\end{equation}
We prove this theorem by considering two  cases.
\\
{\it Case 1. } $ |x| \wedge |y| \geq \sqrt{t}$.   In this case ,
$  p(t,x,y)\geq   p_{(0, \infty)}(t,x,y)\gtrsim t^{-1/2} e^{- {c_3|x-y|^2}/{t}}$.
Thus we have by Proposition \ref{offdiagUBE}, 
\begin{equation*}
\frac{1}{\sqrt{t}}e^{- {c_{3}|x-y|^2}/{t}} \lesssim p(t,x,y)\lesssim \frac{1}{\sqrt{t}}e^{- {c_{4}|x-y|^2}/{t}}.
\end{equation*}

\noindent {\it Case 2. } $0< |x| \wedge  |y| <\sqrt{t}$. 
Without loss of generality, we may and do assume $|y|< \sqrt{t}$.  
By \eqref{e:5.8}, 
\begin{equation}\label{e:5.43}
\int_0^t p(t-s,a^*, y) \IP_x(\sigma_{a^*}\in ds)
= \int_0^t p(t-s,a^*, y)  \frac{|x|}{\sqrt{2\pi s^3}} e^{-x^2/2s} ds.
\end{equation}
By Proposition \ref{P:5.12} and Lemma \ref{L:5.11}
\begin{eqnarray*}
\int_0^{t/2} p(t-s,a^*, y)  \frac{|x|}{\sqrt{2\pi s^3}} e^{-x^2/2s} ds 
&\lesssim&  \int_0^{t/2} \frac{1}{t-s} \left(1+ \frac{|y|\log (t-s) }{\sqrt{t-s}}   \right)   e^{- {y^2}/{2(t-s)}}
 \frac{|x|}{\sqrt{2\pi s^3}} e^{-x^2/2s} ds \\
 &\lesssim &   \frac{1}{t } \left(1+ \frac{|y|\log t }{\sqrt{t }}   \right) \int_0^{t/2}\frac{|x|}{ s^{3/2}} e^{-x^2/2s} ds \\
 &\lesssim &   \frac{1}{t } \left(1+ \frac{|y|\log t }{\sqrt{t }}   \right)  \frac{1}{1+|x|/\sqrt{t}} e^{-x^2/t} \\
 &\lesssim &   \frac{1}{t } \left(1+ \frac{|y|\log t }{\sqrt{t }}   \right) e^{-(x^2+y^2)/t},
 \end{eqnarray*}  
while by Lemma \ref{L:5.16}, 
\begin{eqnarray*}
\int_{t/2}^{t-2} p(t-s,a^*, y)  \frac{|x|}{\sqrt{2\pi s^3}} e^{-x^2/2s} ds 
&\lesssim&  \int_ {t/2}^{t-2} \frac{1}{t-s} \left(1+ \frac{|y|\log (t-s) }{\sqrt{t-s}}   \right)   e^{- {y^2}/{2(t-s)}}
 \frac{|x|}{\sqrt{2\pi s^3}} e^{-x^2/2s} ds \\
 &\lesssim &  \frac{|x|}{t^{3/2}} e^{-x^2/2t} \int_{t/2}^{t-2}  \frac{1}{t-s} 
 \left(1+ \frac{|y|\log (t-s) }{\sqrt{t-s}}   \right)   e^{- {y^2}/{2(t-s)}}ds \\
 &\stackrel{r=t-s}{=} &   \frac{|x|}{t^{3/2}} e^{-x^2/2t} \int_{2}^{t/2}  \frac{1}{r} 
 \left(1+ \frac{|y|\log r  }{\sqrt{r}}   \right)   e^{- {y^2}/{2r}}dr \\ 
 &\lesssim &     \frac{|x|}{t^{3/2}} e^{-x^2/2t}  \log t  \asymp   \frac{|x| \log t }{t^{3/2}} e^{-(x^2+y^2) /2t}.  
 \end{eqnarray*}  
 A similar calculation shows
$$
\int_0^{t/2} p(t-s,a^*, y)  \frac{|x|}{\sqrt{2\pi s^3}} e^{-x^2/2s} ds
\gtrsim \frac{1}{t } \left(1+ \frac{|y|\log t }{\sqrt{t }}   \right) e^{-2(x^2+y^2)/t}
$$
and 
$$
\int_{t/2}^{t-2} p(t-s,a^*, y)  \frac{|x|}{\sqrt{2\pi s^3}} e^{-x^2/2s} ds
\gtrsim   \frac{|x|}{t^{3/2}} e^{-x^2/2t}  \log t  \asymp   \frac{|x| \log t }{t^{3/2}} e^{-2(x^2+y^2) / t}.
$$
 By Theorem \ref{T:4.5}, 
 \begin{eqnarray*}
\int_{t-2}^{t} p(t-s,a^*, y)  \frac{|x|}{\sqrt{2\pi s^3}} e^{-x^2/2s} ds 
&\lesssim & \int_{t-2}^{t} \frac1{\sqrt{t-s}} e^{-c_5 y^2/(t-s)}  \frac{|x|}{\sqrt{2\pi s^3}} e^{-x^2/2s} ds \\
&\lesssim &  \frac{|x|}{\sqrt{  t^3}} e^{-x^2/ t} \int_{t-2}^{t} \frac1{\sqrt{t-s}} e^{-c_5 y^2/(t-s)}  ds \\
&\lesssim &  \frac{|x|}{ t^{3/2}} e^{-x^2/2t} \lesssim  \frac{|x|}{t^{3/2}} e^{-(x^2+y^2)/2t}.
 \end{eqnarray*}
 These estimates together with  \eqref{e:5.41}-\eqref{e:5.43} establish the theorem. 
 \end{proof}
 
 Theorem \ref{T:5.17} together with Theorems \ref{154} and \ref{T:5.15}
gives Theorem \ref{largetime}.

\section{H\"{o}lder regularity of Parabolic Functions}\label{S:6} 

As we noted in Remark \ref{rmkonsmalltimeHKE}(iii),   parabolic Harnack principle fails for the BMVD $X$.
However we show in this section that H\"older regularity holds for the parabolic functions of $X$.
In the elliptic case (that is, for harmonic functions instead of parabolic functions),
 this kind of phenomenon has been observed for solutions of SDEs driven by multidimensional 
 L\'evy processes with independent coordinate processes; 
 see \cite{BC2}.

To show the H\"{o}lder-continuity of parabolic functions of  $X$, we begin with  the following two lemmas.

\begin{lem}\label{L:6.1}
There exist $C_1>0$ and $0<C_2\leq 1/2$ such that for every $x_0\in E$ and $R>0$, 
\begin{equation*}
p_{B(x_0,R)}(t,x, y) \ge \frac{1}{2}\; p(t, x, y) 
\qquad \hbox{for   } t\in (0,  C_1/ (R\vee 1)^2 ] \hbox{ and } x, y\in B(x_0, C_2R) .
\end{equation*}
\end{lem}

\begin{proof}
By  Theorem \ref{T:smalltime}, there exist constants $c_i>0$, $1\le i\le 4$, such that for all $t\leq 1$ and  $x,y \in E$,  
\begin{equation}\label{24105}
\frac{c_3}{\sqrt{t}}e^{- {c_4 \rho (x,y)^2}/{t}} \leq p(t, x,y)\le \frac{c_1}{t}e^{- {c_2\rho (x,y)^2}/{t}}   .
\end{equation}
We choose $0<c_5<1/2$ sufficiently small such that
\begin{equation}\label{24119}
\frac{(1-c_5)^2}{(2c_5)^2}\ge \frac{c_2}{c_4}.
\end{equation}
As $t \mapsto t^{-1}e^{-c_0/t}$ is increasing in $t\in (0, 1/c_0]$, we have  for $0<t\leq 1/(c_2(1-c_5)^2R^2)$
and $x, y\in B(x_0, c_5R)$, 
\begin{eqnarray*}
\overline p_{B(x_0, c_5R)} (t, x, y)&:=& \IE_x [ p(t-\tau_{B(x_0, c_5R)}, X_{B(x_0, c_5R)}, y); \tau_{B(x_0, c_5R)}<t ] \\
& \lesssim  &   \IE_x [ (t-\tau_{B(x_0, c_5R)})^{-1}  e^{- {c_2 ((1-c_5)R)^2}/{(t-\tau_{B(x_0, c_5R)}})}  ; \tau_{B(x_0, c_5R)}<t ]  \\
&\leq &  t^{-1}   e^{- {c_2 ((1-c_5)R)^2}/t} \lesssim e^{- {c_2 ( 1-c_5)^2R^2}/2t} ,  
\end{eqnarray*}
while
\begin{equation*}
 p (t, x, y) \geq \frac{c_3}{\sqrt{t}}e^{- {c_4 (2c_5R)^2}/{t}} \stackrel{\eqref{24119}}{\geq}  \frac{c_3}{\sqrt{t}}e^{- {c_2 ( 1-c_5)^2R^2}/2t}.
\end{equation*}
Hence there is $c_6\leq  1/(c_2(1-c_5)^2 )$ so that $p (t, x, y) \geq \frac12 \overline p_{B(x_0, c_5R)} (t, x, y)$
for every $R>0$, $x_0\in E$, $0< t \leq   c_6 /(R\vee 1)^{2} $ and $x, y\in B(x_0, c_5R)$. This proves the lemma as 
$p_{B(x_0, c_5R)}(t, x, y) = p(t, x, y) -
\overline  p_{B(x_0, c_5R)} (t, x, y)$. 
\end{proof}

Let $Z_s=(V_s, X_s)$ be the space-time process of $X$ where $V_s=V_0+s$.  
 In the rest   of this section, 
 $$
 Q(t,x,R):= (t, t+R^2)\times B_\rho(x,R) .
 $$
For any Borel measurable set $A\subset Q(t,x,R)$, we use $|A|$ to denote its measure under the product measure 
$dt \times m_p (dx)$. 

\begin{lem}\label{L:6.2}
Fix $R_0\geq 1$. There exist  constants $0<C_3\leq 1/2$ and $C_4>0$ such that for all $0<R\leq R_0$,   $x_0\in E$, 
 $x\in  B_\rho (x_0, C_3R)$ and any $A\subset Q(0,x_0,C_3R)$ with
 $\frac{|A|}{|Q(0, x_0, C_3R)|} \ge \frac{1}{3}$,
\begin{equation}
\IP_{(0,x)}(\sigma_A<\tau_R )\ge C_4,
\end{equation}
where $\tau_R=\tau_{Q(0,x_0, R)}=\inf\{t\geq 0: X_t \notin B_\rho(x_0, R)\} \wedge R^2$
and $\sigma_A:=\inf\{t\geq 0: (V_t, X_t)\in A\}$.
\end{lem}

\begin{proof}
Let $C_1$ and $C_2$ be the constants in Lemma \ref{L:6.1}. Define $C_3= (C_1/R_0^3)^{1/2} \wedge C_2$. 
For $x_0\in  E$ and $R\in (0, R_0]$, denote by  $X^{B_\rho(x_0, R)}$  the subprocess of $X$ killed upon exiting the ball $B_\rho(x_0,R)$ and 
 $p_{B_\rho(x_0, R)}$ its transition density with respect to the measure $m_p$.
 As   $  |(0,  C_3^2R^2/6)\times B_\rho (x_0, C_3R)|=|Q(0, x_0, C_3R)|/6$ and $|A|\geq |Q(0, x_0, C_3R)|/3$,
we have 
$$ | \{(t, x)\in A: t\in [C_3^2R^2/6, C_3^2R^2] \}| \geq |Q(0, x_0, C_3R)|/6.
$$ 
For $s>0$, let $A_s:=\{x\in E: (s,x)\in A\}$.  
Note that
\begin{eqnarray}
  \IE_x \int_0^{\tau_R }\mathbf{1}_A(s,X_s)ds
  &= & \IE_x \int_0^{\tau_R \wedge (C_3R)^2} \mathbf{1}_A(s, X^{B_\rho(x_0,R)})ds   \nonumber \\
 &=& \int_0^{C_3^2R^2}\IP_{x}\left(X_s^{B_\rho(x_0,R)}\in A_{s}\right)ds \nonumber \\
 & =& \int_{0}^{C_3^2 R^2}\int_{A_{s}}p_{B_\rho(x_0,R)}(s,x,y)m_p(dy)ds \nonumber \\
 &\geq  &\int_{C_3^2 R^2/6}^{C_3^2 R^2} p_{B_\rho(x_0,R)}(s,x,y)m_p(dy)ds \label{e:6.4}.
 \end{eqnarray}
 We now consider two cases.
 
 \noindent {\it Case 1.} $m_p( B_\rho(x_0, R))>pR/6$. In this case, we have    by \eqref{e:6.4}, 
 Lemma \ref{L:6.1} and \eqref{24105} that 
 \begin{align*}
 \IE_x \int_0^{\tau_R }\mathbf{1}_A(s,X_s)ds
 & \gtrsim \int_{C_3^2 R^2/6}^{C_3^2 R^2}\int_{A_{s}} \frac1{\sqrt{t}} e^{-c_2 \rho (x, y)^2 /t} m_p(dy)ds \nonumber \\
 & \gtrsim  \frac1R   | \{(t, x)\in A: t\in [C_3^2R^2/6, C_3^2R^2] \}|  \nonumber \\
 & \gtrsim    |Q(0, x_0, C_3R)|/ R \gtrsim R^2. 
\end{align*}
 
\noindent{\it Case 2.} $m_p(B_\rho(x_0, R)) \le pR/6$. In this case,  $x_0$ must be in $D_0$ with  $\rho(x_0, a^*)\ge \frac{5R}{6}$ and so
$m_p(B_\rho(x_0, R))\ge (5R/6)^2$. Thus we have by \eqref{e:6.4} and Theorem \ref{T:smalltime}(iii), 
\begin{eqnarray*}
 \IE_x \int_0^{\tau_R }\mathbf{1}_A(s,X_s)ds
 & \gtrsim &\int_{C_3^2 R^2/6}^{C_3^2 R^2}\int_{A_{s}} \frac1t  e^{-c_2 \rho (x, y)^2 /t} m_p(dy)ds \nonumber \\
 & \gtrsim  & \frac1{R^2}   | \{(t, x)\in A: t\in [C_3^2R^2/6, C_3^2R^2] \}|  \nonumber \\
 & \gtrsim   & |Q(0, x_0, C_3R)|/ R^2 \gtrsim R^2.   
\end{eqnarray*}
Thus in both cases, there is a constant $c_0 >0$ independent of $x_0$ and $R\in (0, R_0]$ so that 
\begin{equation}\label{e:6.5}
 \IE_x \int_0^{\tau_R }\mathbf{1}_A(s,X_s)ds \geq c_0 \, R^2.
 \end{equation}
 On the other hand,
\begin{align*}
\IE_x\int_0^{\tau_R }\mathbf{1}_A(s,X_s)ds&= \int_0^\infty \IP_x\left(\int_0^{\tau_R }\mathbf{1}_A(s,X_s)ds>u\right)du\\
&=\int_0^{R^2}
\IP_x \left(\int_0^{\tau_R}\mathbf{1}_A(s,X_s)ds>u\right)du\\
&\le R^2\, \IP_x (\sigma_A<\tau_R).
\end{align*}
The desired estimate now follows from this and \eqref{e:6.5}.
\end{proof}

\begin{thm}
For every $R_0>0$, there are constants $C=C(R_0)>0$ and $\beta \in (0, 1)$ such that for every $R\in (0, R_0]$, $x_0\in E$, and every bounded parabolic function $q$ in $Q(0, x_0, 2R)$, it holds that
\begin{equation}\label{Holdercont}
|q(s,x)-q(t,y)|\le C\|q\|_{\infty, R}\, 
R^{-\beta}\left(|t-s|^{1/2}+\rho(x,y)\right)^\beta 
\end{equation}
for every  $(s,x), \, (t,y)\in Q(0,x_0, R/4)$, where $\displaystyle{\|q\|_{\infty, R}:=\sup_{(t,y)\in (0, 4R^2] \times B_\rho(x_0, 2R) }|q(t,y)|}$.
\end{thm}

\begin{proof}
With loss of generality, assume $0\le q(s)\le \|q\|_{\infty, R}=1$. We first assume $x_0=a^*$ and show that \eqref{Holdercont} holds 
for all $(s,x), (t,y)\in Q(0, x_0, R)$ (instead of $Q(0, x_0, R/4)$). Let $C_3\in (0, 1/2]$ and $C_4\in (0, 1)$ be the constants in Lemma \ref{L:6.2}. Let 
$$
\eta=1-C_4/4 \geq 3/4\quad \hbox{and} \quad  \gamma=C_3/2\leq 1/4.
$$
 Note that for every $(s,x)\in Q(0, a^*, R)$, $q$ is parabolic in $Q(s,x,R)\subset Q(0, a^*, 2R)$. We will show by induction that $\displaystyle{\sup_{Q(s,x,\gamma^kR)}|q|-\inf_{Q(s,x,\gamma^kR)}|q|\le \eta^k}$ for all integer $k$. For notation convenience, we denote $Q(s,x,\gamma^kR)$ by $Q_k$.
Define $a_i=\displaystyle{\inf_{Q_i}q}$, $b_i=\displaystyle{\sup_{Q_i}q}$. Clearly, $b_i- a_i\le 1\le \eta^i$ for all $i\le 0$. Now suppose $b_i-a_i\le \eta^i$ for all $i\le k$ and we wil show that $b_{k+1}-a_{k+1}\le \eta^{k+1}$. Observe that $Q_{k+1}\subset Q_k$ and so $a_k\le q\le b_k$ on $Q_{k+1}$. Define 
$$
A':=\left\{z\in \big( s+(\gamma^{k+1}R)^2, s+  (C_3 \gamma^kR)^2 \big)
\times B_\rho (x, C_3 \gamma^kR):  q(z)\le (a_k+b_k)/2\right\},
$$ 
which is a subset of $Q_k$.
Note that 
$$
\left|\big(s+(\gamma^{k+1}R)^2, s+ (C_3 \gamma^kR)^2\big)
\times B_\rho (x, \gamma^kR) \right|= \frac34 (C_3\gamma^kR)^2 \, m_p (B_\rho(x, C_3\gamma^k R)).
$$
We may suppose $|A'|\geq \frac12 
(C_3\gamma^kR)^2 \, m_p (B(x, C_3 \gamma^k R))$;
otherwise we consider $1-q$ instead of $q$. Let $A$ be a compact subset of $A'$ such that $|A|\ge \frac12 (C_3\gamma^kR)^2 \, m_p (B(x, C_3 \gamma^k R))$. For any given $\eps>0$, pick $z_1=(t_1, x_1), z_2\in Q_{k+1}$ so that $q(z_1)\ge b_{k+1}-\eps$ and $q(z_1)\le a_{k+1}+\eps$. Note that  $Z_{\tau_{k}}\in \partial Q_{k}$ as 
BMVD $X_t$ has continuous sample paths. So by
Lemma \ref{L:6.2},  
\begin{eqnarray*}
b_{k+1}-a_{k+1}-2\eps
&\le &q(z_1)-q(z_2) \\
&=& \IE_{z_1}\left[q(Z_{\sigma_A\wedge \tau_{k}})-q(z_2)\right]  \\
&=& \IE_{z_1}\left[q(Z_{\sigma_A})-q(z_2); \sigma_A< \tau_k\right]+\IE_{z_1}\left[q(Z_{\tau_k})-q(z_2); \tau_k < \sigma_A \right]\\
\\
&\le& \left(\frac{a_k+b_k}{2}-a_k\right)\IP_{z_1} \left(\sigma_A<\tau_{k}) \right)
+(b_k-a_k)\IP_{z_1}(\sigma_A>\tau_k)  \\
&=& (b_k-a_k)\left( 1-\IP_{z_1}(\sigma_A<\tau_{k})/2\right) 
\\
&\le & \eta^k(1-C_4/2)
\\
&\le & \eta^{k+1}.
\end{eqnarray*}
Since $\eps$ is arbitrary, we get $b_{k+1}-a_{k+1}\leq \eta^{k+1}$. 
This proves that $b_{k}-a_{k}\leq \eta^{k}$ for all integer $k$.

For $z=(s,x)$ and $w=(t,y)$ in $Q(0,a^*,R)$ with $s\le t$, let $k$ be the smallest integer such that $|z-w|:=|t-s|^{1/2}+\rho(x,y)\le \gamma^kR$. 
Then
\begin{equation*}
|q(z)-q(w)|\le \eta^k=\gamma^{k\log\eta/\log \gamma}
\le \left(\frac{|z-w|}{\gamma R}\right)^{\log\eta/\log\gamma}.
\end{equation*}
This establishes 
\eqref{Holdercont} for  $x_0=a^*$ and for every $(s,x), (t, y)\in Q(0,a^*, R)$ with $\beta=\log \eta /\log \gamma$. Note that 
$\beta \in (0, 1)$ since $0< \gamma <\eta <1$.

For general $x_0\in E$, we consider two cases based on the distance $\rho(x, a^*)$:

\smallskip

\noindent {\it Case 1.} $|x|_\rho<R/2$. In this case, 
$Q(0, x_0, R/4)\subset Q(0, a^*, 3R/4) \subset Q(0,a^*, 3R/2)\subset Q(0,x_0, 2R)$. By  what we have 
established above, 
\begin{equation*}
|q(s,x)-q(t,y)|\le C(R_0)\|q\|_{\infty, R} \,
R^{-\beta}\left(|t-s|^{1/2}-\rho(x,y)\right)^\beta \quad \hbox{for } (s,x), (t,y)\in Q(0,x_0, R/4).
\end{equation*}

\smallskip

\noindent {\it Case 2. }$|x|_\rho\ge R/2$. Since $a^*\notin Q(0, x_0, R/2)$, it follows from the classical results for Brownian motion in $\IR^d$
with $d=1$ and $d=2$ that for every $(s,x), (t,y)\in Q(0, x_0, R/4)$
\begin{align*}
|q(s,x)-q(t,y)|&\le C(R_0) \| q \|_{\infty, R/2} R^{-\beta}\left(|t-s|^{1/2}+\rho(x,y)\right)^\beta\\
&\le  C(R_0)\|q \|_{\infty, R} \,\, R^{-\beta}\left(|t-s|^{1/2}+\rho(x,y)\right)^\beta .
\end{align*}
This  completes the proof of the theorem.
\end{proof}

\section{Green Function Estimates}

\indent In this section, we establish two-sided bounds for the Green function of BMVD 
$X$ killed upon exiting a bounded connected $C^{1,1}$ open set $D\subset E$. 
Recall that the Green function $G_D(x,y)$ is defined as follows:
\begin{equation*}
G_D(x,y)=\int_0^\infty p_D(t,x,y)dt,
\end{equation*}
where $p_D(t, x, y)$ is the transition density function of the subprocess $X^D$ with respect to $m_p$. 
 We assume $a^*\in D$ throughout this section, as otherwise, due to the connectedness of $D$, either $D\subset \IR_+$ or $D\subset D_0$. Therefore $G_D(x,y)$ is just the standard Green function of a bounded $C^{1,1}$ domain  for Brownian motion 
 in one-dimensional or two-dimensional  spaces, whose two-sided estimates are known, see \cite{CZ}.  
It is easy to see from 
$$ p_D(t, x, y)=p(t, x, y)-\IE_x [ p(t-\tau_D, X_{\tau_D}, y); \tau_D <t],
$$
that $p_D(t, x, y)$ is jointly continuous in $(t, x, y)$.

 Recall that for any bounded open set $U\subset E$,   $\delta_U(\cdot):=\rho(\cdot, \partial U)$  denotes the $\rho$-distance to
 the boundary $\partial U$. 
 For notational convenience, we set $D_1:=D\cap (\IR_+\setminus \{a^*\})$ and $D_2:=D\cap D_0$.
 Note that  $a^*\in \partial D_1 \cap \partial D_2$.
 
 \medskip

The following theorem gives two-sided Green function estimates for $X$ in bounded $C^{1,1}$ domains.

\begin{thm} \label{T:7.1} 
Let $G_{_D}(x,y)$ be the Green function of $X$ killed upon exiting $D$, where $D$ is a connected bounded $C^{1,1}$ domain of $E$ containing 
$a^*$.   We have for $x\not= y$ in $D$, 
\begin{equation*}
G_{_D}(x,y)\asymp\left\{
  \begin{array}{ll}
    \delta_D(x)\wedge \delta_D(y),
&
\hbox{$x\in D_1 \cup \{a^*\}$, $y\in D_1 \cup \{a^*\}$;}
\\ \\
  \delta_D(x)  \delta_D(y)+\ln\left(1+\frac{\delta_{D_2}(x)\delta_{D_2}(y)}{|x-y|^2}\right),
& \hbox{$x\in D_2$, $y\in D_2$;}
\\ \\
  \delta_D(x) \delta_D(y),
& \hbox{$x\in D_1\cup \{a^*\}$, $y\in D_2$.}
  \end{array}
\right.
\end{equation*}
\end{thm}
 
 \begin{proof} We first show that $G_D(x, a^*)$ is a bounded positive continuous function on $D$. 
By Theorem \ref{T:smalltime}, there is a constant $c_1>0$ so that for every $x\in D$, 
\begin{eqnarray}
\IP_x(\tau_D<1)&\ge& \IP_x(X_1\in E\setminus  D) =\int_{\IR_+ \cap  D^c} p(1,x,z)m_p(dz)
+ \int_{D_0 \cap  D^c} p(1,x,z)m_p(dz) \nonumber  \\
&\ge & c_1.   \label{e:7.1}
\end{eqnarray}
Thus $\IP_x (\tau_D\geq 1) \leq 1-c_1$ for every $x\in D$. 
By the strong Markov property of $X$, there are constants $c_2, c_3>0$ so that 
$\IP_x (\tau_D\geq t) \leq c_2 e^{-c_3t} $ for every $x\in D$
and  $t>0$. For $t\geq 2$ and $x, y\in D$, we thus have by Theorem \ref{T:smalltime},
\begin{eqnarray*}
p_{_D}(t,x,y) =\int_D p_{_D}(t-1, x,z)p_{_D}(1,z,y)m_p(dz) \leq 
c_4 \int_D p_{_D}(t-1, x,z) m_p(dz) \leq c_5e^{-c_3t}.
\end{eqnarray*}
By Theorem \ref{T:smalltime} again, we conclude that 
$$ 
G_D(x, a^*)=\int_0^2 p_D(t, x, y) dt + \int_2^\infty p_D(t, x, y)dt
$$
converges and is a bounded positive continuous function in $x\in D$.
In particular, $G_D(a^*, a^*)<\infty$. 

We further note that $x\mapsto G_D(x, a^*)$ is a harmonic function in $D_1$ and so it is a linear function.
As it vanishes
at $b :=\partial D \cap \IR_+$, we have 
\begin{equation}\label{e:7.2}
G_D(x, a^*) = c_6 |b-x|\asymp  \delta_D (x) \quad  \hbox{for } x\in D_1.
\end{equation}

\medskip

\noindent (i) Assume $x, y \in D_1\cup \{a^*\}$ and $x\not= y$.  If $x=a^*$ or $y=a^*$, 
 the desired estimate holds in view of \eqref{e:7.2}.
Thus we assume neither $x$ nor $y$ is $a^*$.  By the strong Markov property of $X$, 
$$ 
G_D(x, y )= G_{D_1}(x, y) + \IE_x [ G_D(X_{\sigma_{a^*}}, y); \sigma_{a^*}<\tau_D]
=G_{D_1}(x, y) + G_D(a^*, y)   \IP_x ( \sigma_{a^*}<\tau_D).
$$
We know from the one-dimensional Green function estimates,
$$
G_{D_1}(x, y) \asymp \delta_{D_1}(x)\wedge \delta_{D_1}(y) . 
$$ 
On the other hand, $x\mapsto \IP_x ( \sigma_{a^*}<\tau_D)$ is a harmonic function
in  $D_1$ that vanishes at $b$. Thus by the same reasoning as that for \eqref{e:7.2}, we have 
\begin{equation}\label{e:7.3}
\IP_x ( \sigma_{a^*}<\tau_D)= c_7 |x-b|\asymp  \delta_D (x) \quad  \hbox{for } x\in D_1.
\end{equation}
It follows then
$$ 
G_D(x, y ) \asymp \delta_{D_1}(x)\wedge \delta_{D_1}(y) + \delta_{D }(x) \delta_{D }(y)
\asymp  \delta_{D }(x)\wedge \delta_{D }(y) .
$$

\noindent (ii) Assume that $x, y\in D_2$. 
By the strong Markov property of $X$, 
$$ 
G_D(x, y )= G_{D_2}(x, y) + \IE_x [ G_D(X_{\sigma_{a^*}}, y); \sigma_{a^*}<\tau_D]
=G_{D_2}(x, y) + G_D(a^*, y)   \IP_x ( \sigma_{a^*}<\tau_D).
$$
Since both $y\mapsto G_D(a^*, y)$ and $x\mapsto \IP_x ( \sigma_{a^*}<\tau_D)$
are bounded positive harmonic functions on $D\cap D_0$ that vanishes on $D_0 \cap \partial D$,
it follows from the boundary Harnack inequality for Brownian motion in $\IR^2$ that 
\begin{equation}\label{e:7.4}
 G_D(a^*, y) \asymp \delta_D (y) \quad \hbox{and} \quad 
 \IP_x ( \sigma_{a^*}<\tau_D) \asymp \delta_D (x)
 \end{equation}
This combined with the Green function estimates of $G_{D_2}(x, y)$ (see \cite{CZ}) yields 
$$
G_D(x, y) \asymp 
\ln \left(1+\frac{\delta_{D_2}(x)\delta_{D_2}(y)}{|x-y|^2}\right) +\delta_D(x)\delta_D (y).
$$

\noindent
(iii) We now consider the last  case that $x\in D_1\cup \{a^*\}$ and $y\in D_2$.
When $x=a^*$, the desired estimates follows from \eqref{e:7.4} and so it remains to
consider $x\in D_1$ and $y\in D_2$.  
By the strong Markov property of $X$, \eqref{e:7.3} and \eqref{e:7.4},
\begin{eqnarray*}
G_D(x, y) = \IE_x [ G_D(X_{\sigma_{a^*}} , y); \sigma_{a^*}<\tau_D]
= G_D(a^*, y) \IP_x (\sigma_{a^*}<\tau_D) \asymp \delta_D(x) \delta_D (y).
\end{eqnarray*}
This completes the proof of the theorem.  
\end{proof}

\section{Some other BMVD}\label{S:8} 

In this section, we present two more examples of spaces with varying dimension
that are variations of  $\IR^2 \cup \IR$ considered in previous sections. 
The existence and uniqueness of BMVD on these spaces can be established
in a similar way as Theorem \ref{existence-uniqueness} in Section \ref{S:2}. We will concentrate on short time two-sided
estimates on the transition density function of these two BMVD.
One can further study their large time heat kernel estimates. 
Due to the space limitation, we will not pursue
it  in this paper.

\subsection{A square with several flag poles}\label{S:8.1}

In this subsection, we study the BMVD on a large square with multiple flag poles of infinite length. 
The state space $E$ of the process embedded in $\IR^3$ is defined as follows.
Let $\{z_j; 1\leq j\leq k\}$ be $k$ points in $\IR^2$ that have distance 
at least 4 between each other. Fix a finite sequence $\{\eps_j;  1\leq j \leq k \} \subset (0,1/2)$ and a sequence of positive constants 
$\vec p:=\{ p_j; 1\leq j\leq k\}$.
For $1\le j\le k$, denote by  $B_j$  the closed disk on $\IR^2$ centered at $z_j$ with radius $\eps_j$. Clearly, the distance between two distinct balls is at least 3. 
Let $\displaystyle{D_0=\IR^2\setminus  \left(\cup_{1\le i\le k}B_i\right)}$. For $1\le j\le k$, denote by $L_j$ the half-line $\{(z_j, w)\in \IR^3: w>0\}$. By identifying each closed ball $B_j$ with a singleton denoted 
by $a_j^*$, we equip the space $E:=D_0\bigcup \{a_1^*, \cdots, a_k^*\}\bigcup (\cup_{i=1}^k L_i)$ with induced topology from $\IR^2$ and the half-lines $L_j$, $1\leq j\leq k$, 
with the endpoint of the half-line $L_i$ identified with $a_i^*$ and
a neighborhood of $a_i^*$ defined as $\{a_i^*\}\cup \left(U_1\cap L_i\right)\cup \left(U_2\cap D_0\right)$ for some neighborhood $U_1$ of $0$ in $\IR$ and $U_2$ of $B_i $ in $\IR^2$. Let $m_{\vec p}$ be the measure on $E$ whose restriction on $D_0$ is the two-dimensional Lebesgue measure, and whose restriction on $L_j$ is the one-dimensional Lebesgue measure multiplied by $p_j$
for $1\leq j\leq k$.  
So in particular, $m_{\vec p} \left(\{a_i^*\}\right)=0$ for all $1\le i\le k$. We denote the geodesic distance on $E$ by $\rho$. 

Similar to Definition \ref{D:1.1}, BMVD on the plane with multiple half lines is defined as follows.

\begin{dfn}\rm Given a finite sequence $\vec \eps:=\{\eps_j; 1\leq j\leq k\}\subset (0,1/2)$ and a sequence of positive constants $\vec p:=\{p_j; 1\leq j\leq k\}$. A Brownian motion with varying dimension with parameters $(\vec \eps, \vec p)$ on $E$  is an $m_{\vec p}$-symmetric diffusion $X$ on $E$ such that
\begin{description}
\item{(i)} its subprocess in $L_i$, $1\le i\le k$, or $D_0$ has the same law as that of standard Brownian motion in $\IR_+$ or $D_0$;
\item{(ii)} it admits no killings on any $a_i^*$, $1\le i\le k$.
\end{description}
\end{dfn}

Recall that the endpoint of the half-line $L_i$ is identified with $a^*_i$, $1\le i\le k$. Similar to Theorem \ref{existence-uniqueness}, we have the following theorem stating the existence and uniqueness of the planary BMVD $X$ with multiple half lines.

\begin{thm}\label{existence-uniqueness-planary}
For each $k\geq 2$, 
every $\vec\eps :=\{ \eps_j ; 1\leq j\leq k\}\subset (0, 1/2)$ and 
$\vec p:=\{p_j; 1\leq j\leq k\}\subset (0, \infty)$,
BMVD $X$ on $E$ with parameter $(\vec \eps, \vec p)$ exists and is unique.
Its associated Dirichlet form $(\EE, \FF)$ on $L^2(E; m_p)$ is given by
\begin{eqnarray*}
\FF &=& \left\{f: f|_{\IR^2}\in W^{1,2}(\IR^2); \,  f|_{L_i}\in W^{1,2}(\IR),\, f|_{B_i}=f|_{L_i}(a^*_i) \hbox{ for } 1\leq i \leq k \right\}, 
\\ 
\EE(f,g) &=& \frac{1}{2}\int_{D_0} \nabla f(x) \cdot \nabla g(x) dx
 +\sum_{i=1}^k\frac{p_i}{2}\int_{L_i}f'(x)g'(x)dx. 
\end{eqnarray*}
\end{thm}

It is not difficult to see that BMVD $X$ has a continuous 
transition density $p(t,x,y)$ with respect to the measure $m_{\vec p}$. 

\begin{prop}\label{P8:3}
There exist constants $C_1, C_2>0$ such that
\begin{displaymath}
p(t,x,y)\le C_1\left(\frac{1}{t}+
\frac{1}{t^{1/2}}\right)e^{-C_2\rho(x,y)^2/t}
\qquad\text{for all } x, y \in
E \hbox{ and } t> 0 .
\end{displaymath}
\end{prop}

\begin{proof}
By an exactly the same argument as that for Proposition \ref{NashIneq}, 
we can establish Nash-type inequality for $X$. From it, 
the off-diagonal upper bound can be derived using Davies' method
as in Propositions \ref{P:3.2} and \ref{offdiagUBE}. 
\end{proof}

The following theorem gives two-sided bounds for the transition density function $p(t,x,y)$ when $t\in (0, T]$ for each fixed $T>0$.
  
\begin{thm}\label{small-time-multi-line}
Let $T\geq 2$ be fixed. There exist positive constants $C_i$, $3\le i\le 16$ so  that the transition density $p(t,x,y)$ of BMVD $X$ on $E$ satisfies the following estimates when $t\in (0, T]$.
\\
Case 1. For   $x, y \in D_0\cap  B_\rho(a^*_i, 1)$ for some $1\leq i\leq k$,  
\begin{align*}
\frac{C_3}{\sqrt{t}}e^{-C_4\rho(x,y)^2/t}&+\frac{C_3}{t}\left(1\wedge
\frac{\rho(x, a^*_i)}{\sqrt{t}}\right)\left(1\wedge
\frac{\rho(y, a^*_i)}{\sqrt{t}}\right)e^{-C_5|x-y|^2/t} \le p(t,x,y)
\\
&\le \frac{C_6}{\sqrt{t}}e^{-C_7\rho(x,y)^2/t}+\frac{C_6}{t}\left(1\wedge
\frac{\rho(x, a^*_i)}{\sqrt{t}}\right)\left(1\wedge
\frac{\rho(y, a^*_i)}{\sqrt{t}}\right)e^{-C_{8}|x-y|^2/t};
\end{align*}
 Case 2. For some $1\le i\le k$, $x\in L_i,\, y\in L_i\cup B_\rho(a^*_i, 1)$,
 \begin{align*}
  \frac{C_{9}}{\sqrt{t}}e^{-C_{10}\rho(x,y)^2/t} \le p(t,x,y) \le \frac{C_{11}}{\sqrt{t}}e^{-C_{12}\rho(x,y)^2/t};  
\end{align*}
Case 3. For all other cases, 
\begin{equation}\label{e:8.1}
  \frac{C_{13}}{t}e^{-C_{14}\rho(x,y)^2/t} \le p(t,x,y)\le \frac{C_{15}}{t}e^{-C_{16}\rho(x,y)^2/t}.
\end{equation}
\end{thm}

\begin{proof}
The idea of the proof is to reduce it  to the heat kernel for BMVD on plane
with one vertical half-line. 
For an open subset $D$ of $E$, we use $X^D$ to denote the subprocess of 
$X$ killed upon leaving $D$ and $p_D(t, x, y)$ the transition density function of $X^D$ with respect to $m_{\vec p}$. 

Let $C_1>0$ and $C_2\in (0, 1/2)$ be the constants in Lemma \ref{L:6.1}.

(i) We first show that the desired estimates hold for any $x, y\in E$ with 
$\rho (x, y) <2C_2$ and for every $t\in (0, T]$.
In this case, let $z_0\in E$ so that
$\{x, y\} \in B_\rho (z_0, C_2)$. 
Since $\rho (a^*_i, a^*_j)>3$ for $i\not= j$, 
without loss of generality we may and do assume that $a^*_1$ is the base that is closest to
$z$ and so $\min_{2\leq j\leq k} \rho (z_0, a^*_j)>3/2$.
We have by Lemma \ref{L:6.1} that
\begin{equation}\label{e:8.2}
 p_{B_\rho (z_0, 1)}(t,w,z) \geq \frac12 p_1 (t, w, z)
\quad \hbox{ for } t\in (0, C_1] \hbox{ and } w, z\in B_\rho (z_0, C_2),
\end{equation} 
where $p_1 (t, x, y)$ stands for the transition density function
of BMVD on the plane with one vertical halfline $L_1$ at base $a_1^*$ and vertical half line $L_1$. This together with Theorem \ref{T:smalltime} in particular implies that 
there is a constant $c_1>0$ so that 
$$
p_{B_\rho (z_0, 1)}(t,w,z) \geq c_1 
\quad \hbox{for every } t\in [C_1/2, C_1] \hbox{ and } 
w, z \in B_\rho (z_0, C_2).
$$
It thus follows from the Chapman-Kolmogorov's equation
that there is a constant $c_2>0$ so that 
\begin{equation}\label{e:8.3}
p_{B_\rho (z_0, 1)}(t,w,z) \geq c_2 
\quad \hbox{for every } t\in [C_1, T] \hbox{ and } 
w, z \in B_\rho (z_0, C_2).
\end{equation} 
This together with \eqref{e:8.2} implies that there is a constant 
$c_3>0$ so that 
\begin{equation}\label{e:8.4}
 p(t, w, z) \geq 
p_{B_\rho (z_0, 1)}(t,w,z) \geq c_3 p_1 (t, w, z)
\quad \hbox{ for } t\in (0, T] \hbox{ and } w, z\in B_\rho (z_0, C_2).
\end{equation}
On the other hand, we have by Proposition \ref{offdiagUBE} and the fact that
$s\mapsto s^{-1}e^{-a^2/s}$ is increasing in $(0, a^2)$ and decreasing
in $(a^2, \infty)$ that for $t\in (0, T]$ and $\rho (x, y)<2C_2<1$, 
\begin{eqnarray*}
\bar{p}_{B_\rho (z_0, 1)}(t,x,y)&:=&
\IE_x[p(t-\tau_{B_\rho (z_0, 1)},X_{\tau_{B_\rho (z_0, 1)}},y); \tau_{B_\rho (z_0, 1)}< t]  \\
&\lesssim& t^{-1}e^{-c_4/t}\lesssim e^{-c_5/t}\lesssim e^{-c_6\rho(x,y)^2/t}.
\end{eqnarray*}
Consequently,
\begin{equation} \label{e:8.5} 
p(t, x, y)= {p}_{B_\rho (z_0, 1)}(t,x,y) + \bar{p}_{B_\rho (z_0, 1)}(t,x,y)
\lesssim p_1 (t, x, y)+ e^{-c_6\rho(x,y)^2/t}.
\end{equation}
This together with \eqref{e:8.4} and Theorem \ref{T:smalltime} establishes
the desired estimate of the theorem for any $x, y\in E$ with $\rho (x, y)< 2C_2$
and $t\in (0, T]$. 

\medskip

(ii)   We now consider the case when $x\in L_i$ and  $ y\in L_i\cup B_\rho(a^*_i,   C_2)$ for some $1\le i\le k$  with 
$\rho (x, y)\geq 2C_2$. 
Without loss of generality,  we may and do assume  $i=1$ and $\rho(x, a_1^*)<\rho(y, a_1^*)$ if $y\in L_1$. 
Let $z_0\in L_1$ with $\rho(z_0, a_1^*)=\rho(y, a^*_1)+C_2.$ 
By the strong Markov property of $X$,   \eqref{e:8.4}-\eqref{e:8.5} and Theorem \ref{T:smalltime},  we have for $t\in (0, T]$, 
\begin{eqnarray*}
p(t, x, y) &=& \IE_x [ p(t-\sigma_{z_0}, z_0, y); \sigma_{z_0}<t] \\
&\geq&  c_3 \IE_x [ p_1(t-\sigma_{z_0}, z_0, y); \sigma_{z_0}<t] 
\\
&=& c_3 p_1 (t, x, y)   \gtrsim  t^{-1/2}  \, e^{-c_7\rho (x, y)^2/y}, 
\end{eqnarray*}
and
\begin{eqnarray*}
p(t, x, y) &= &\IE_x [ p(t-\sigma_{z_0}, z_0, y); \sigma_{z_0}<t] \\
&\stackrel{\eqref{e:8.5}}{\lesssim}&   \IE_x [ p_1(t-\sigma_{z_0}, z_0, y); \sigma_{z_0}<t] + \IE_x [ e^{-c_6 C_2^2/(t-\sigma_{z_0})}; \sigma_{z_0}<t]\\
&\leq &  p_1 (t, x, y) + e^{-c_6 C_2^2/t} \IP_x (\sigma_{z_0}<t) \\
&\lesssim &  t^{-1/2} \, e^{-c_8\rho (x, y)^2/y} + e^{-c_6 C_2^2/t}  e^{-|x-z_0|^2/t} \\
&\lesssim &  t^{-1/2} \,  e^{-c_9\rho (x, y)^2/y}.
\end{eqnarray*}
In the second to last inequality, we   used crossing estimate for one-dimensional Brownian motion and Lemma \ref{L:5.11}. 
The above two estimates  give the desired estimates. 

\medskip

(iii) Let $D_1= D_0\setminus  \cup_{j=1}^k  B_\rho (a^*_j, C_2) $.
There are three remaining cases:  

\begin{description}
\item{(a)}  $x\in L_i \cup B_\rho (a^*_i, C_2)$ and $y\in L_j \cup B_\rho (a^*_j, C_2)$ for $i\not= j$;

\item{(b)} $x\in L_i \cup B_\rho (a^*_i, C_2)$ for some $1\leq i \leq k$ and $y\in D_1 $
with $\rho (x, y)\geq 2C_2$;    

\item{(c)} $x, y\in D_1 $ with $\rho (x, y)\geq 2C_2$.
\end{description}

We claim that \eqref{e:8.1} holds for  all these three cases. The upper bound in \eqref{e:8.1} holds due to 
Proposition \ref{offdiagUBE} so it remains to establish the lower bound. 

It follows from the Dirichlet heat kernel estimate for Brownian motion in $C^{1,1}$-domain \cite{Z3, CKP}
that in case (c), for any $t\in (0, T]$, 
\begin{equation} \label{e:8.6} 
p(t, x, y)\geq p_D(t, x, y) \gtrsim t^{-1} e^{-c_{10} |x-y|^2/t} \geq t^{-1} e^{-c_{10} \rho (x, y)^2/t}. 
\end{equation}

For case (b), without loss of generality, we assume $i=1$. Define $u_1(w)= -\rho (w, a^*_1)$ for $w\in L_1$
and $u_1(w)=\rho (w, a^*_1)$ analogous to \eqref{848}. Let $Y=u(X)$ and $\tau_1:=\inf\{t>0: Y_t \geq 2C_2\}$.
Then $Y^{(-\infty, 2C_2)}_t := Y_t$ for $ t\in [0, \tau_1)$,  and $Y^{(-\infty, 2C_2)}_t :=\partial$ for $t\geq \tau_1$
has the same distribution as the killed radial process radial process $Y$ in Section \ref{S:4} for BMVD on plane with
one vertical half-line. By Proposition \ref{P:4.3} and the arguments similar to Proposition \ref{P:4.4} of this paper
and that of \cite{CKP}, one can show that $Y^{(-\infty, 2C_2)}_t$ has a transition density function $p^0(t, w, z)$
with respect to the Lebesgue measure on $\IR$ 
and it has the following two-sided estimates:
$$ 
t^{-1/2} \left( 1\wedge \frac{|w|}{\sqrt{t}} \right) \left( 1\wedge \frac{|z|}{\sqrt{t}} \right) e^{-c_{11} |w-z|^2/t}
\lesssim p^0(t, w, z) 
\lesssim t^{-1/2} \left( 1\wedge \frac{|w|}{\sqrt{t}} \right) \left( 1\wedge \frac{|z|}{\sqrt{t}} \right) e^{-c_{12} |w-z|^2/t}
$$
for $t\in (0, T]$ and $x, y\in (-\infty, 2C_2)$. Thus we have for $t\in (0, T]$, 
$x\in L_1 \cup B_\rho (a^*_1, C_2)$   and $y\in D_1 $
with $\rho (x, y)\geq 2C_2$,
\begin{eqnarray}
p(t, x, y) &\geq& \int_{D_0 \cap (B_\rho (a^*_1, 6C_2/4) \setminus B_\rho (a^*_1, 5C_2/4))} p(t/2, x, z) p(t/2, z, y) 
m_{\vec p} (dz) \nonumber \\
&\gtrsim &  t^{-1} e^{-c_{13} \rho (a^*_1, y)^2/t}  \int_{D_0 \cap (B_\rho (a^*_1, 6C_2/4) \setminus B_\rho (a^*_1, 5C_2/4))}
  p(t/2, x, z) m_{\vec p} (dz)  \nonumber \\
&\geq & t^{-1} e^{-c_{13} \rho (a^*_1, y)^2/t}  \int^{ 6C_2/4}_{5C_2/4} 
  p^0 (t/2, u_1( x) , w)  dw   \nonumber \\
&\gtrsim & t^{-1} e^{-c_{13} \rho (a^*_1, y)^2/t}  \int^{ 6C_2/4}_{5C_2/4} t^{-1/2} 
  e^{-c_{11} (6C_2/4  - u_1(x))^2/t}  dw  \nonumber \\
  &\gtrsim & t^{-1} e^{-c_{13} \rho (a^*_1, y)^2/t} e^{-c_{14} (6C_2/4  - u_1(x))^2/t}  \nonumber \\
  &\gtrsim & t^{-1} e^{-c_{15} \rho (x, y)^2/t} , \label{e:8.7}
  \end{eqnarray} 
where in the second inequality it was used  that $\rho(z,y)\le 5\rho(a^*_1, y)/2$, and in the last two inequalities we used the fact that $ |6C_2/4  - u_1(x)|\geq C_2/2$ and $\rho (x, y)\geq 2C_2$. 

Now for case (a) when $x\in L_i \cup B_\rho (a^*_i, C_2)$ and $y\in L_j \cup B_\rho (a^*_j, C_2)$ for $i\not= j$,
let $z_0\in D_0$ so that both $\rho (z, a^*_i)$ and $\rho (z, a^*_j)$ take values within $(\rho (a^*_i, a^*_j)/3, 2\rho (a^*_i, a^*_j)/3)$.  
We then have by \eqref{e:8.7} that for all $t\in (0, T]$, 
\begin{eqnarray*}
p(t, x, y) &\geq& \int_{D_0\cap B_\rho (z_0, C_2)} p(t/2, x, z) p(t/2, z, y) m_{\vec p} (dz) \nonumber \\
&\gtrsim &  t^{-1} e^{-c_{16} \rho (x, z_0)^2/t }\, t^{-1} e^{-c_{16} \rho (y, z_0)^2/t} \, m_{\vec p} (D_0\cap B_{\rho} (z_0, C_2)) \nonumber \\
&\gtrsim & t^{-1}e^{-c_{17} \rho (x, y)^2/t },
\end{eqnarray*}
where in the last inequality, we used the fact that $\rho (x, y)\geq 3$.  
This completes the proof that the lower bound in \eqref{e:8.1} holds for all three cases (a)-(c).
The theorem is now proved. 
\end{proof}

 \subsection{A large square with an arch}

In this subsection, we study   Brownian motion on a large square  with an arch. 
The state space $E$ of the process is defined as follows. Let $z_1, z_2 \in \IR^2$ with $|z_1-z_2|\geq 6$. 
Fix constants $0<\eps_1, \eps_2<1/2$ and $p>0$. For $i=1,2$, denote by  $B_i$  the closed disk on $\IR^2$ centered at $z_i$
 with radius $\eps_i$. Let $\displaystyle{D_0=\IR^2\setminus  \left(B_1 \cup B_2 \right)}$.
 We short $B_i$ into a   singleton denoted by $a_i^*$. Denote by $L$ a one dimensional arch with two endpoints  $a^*_1$ and $a^*_2$.
 Without loss of generality, we assume $L$ is isometric to an closed interval $[-b, b]$ for some $b\geq 4$. 
 We equip the space $E:= D_0 \cup \{a^*_1, a^*_2\} \cup L$ with the Riemannian distance $\rho$ induced from $D_0$ and $L$, analogous to the  last example  of a large square with multiple flag poles. 
  Let $m_p$ be the measure on $E$ whose restriction on $L$ and $D_0$ is the arch length measure and the Lebesgue measure multiplied by $p$ and $1$, respectively.  In particular, we have $m_p\left(\{a^*_1, a^*_2\}\right)=0$. 
As before,  BMVD on $E$   is defined as follows.

\begin{dfn}\label{def-planary-arch}\rm Given $0<\eps_1, \eps_2<1/2$ and $p>0$, 
BMVD on $E$ with parameters $(\eps_1, \eps_2, p)$ on $E$  is an $m_p$-symmetric diffusion $X$ on $E$ such that
\begin{description}
\item{(i)} its subprocess process in $L$ or $D_0$ has the same distribution as  the one-dimensional or two-dimensional Brownian motion
in $L$ or $D_0$, respectively.
\item{(ii)} it admits no killings at  $\{a_1^*, a_2^*\}$.
\end{description}
\end{dfn}

Similar to Theorem \ref{existence-uniqueness}, we have the following.

\begin{thm}\label{T:8.6}
For every $0<\eps_1, \eps_2 < 1/2$ and 
$p>0$,
BMVD $X$ on $E$ with parameter $(\eps_1, \eps_2, p)$  exists and is unique.
Its associated Dirichlet form $(\EE, \FF)$ on $L^2(E; m_p)$ is given by
\begin{eqnarray*}
  \FF &=& \left\{f: f|_{\IR^2}\in W^{1,2}(\IR^2),\,  f|_{L}\in W^{1,2}(L),\, f|_{B_i}=f|_L(a^*_i)\,,i=1,2  \right\}, \\ 
 \EE (f, g) &=&  \frac{1}{2}\int_{D_0} \nabla f(x) \cdot \nabla g(x) dx+ \frac{p}{2}\int_{L} f'(x)g'(x)dx.
\end{eqnarray*} 
\end{thm}

It is easy to see that BMVD $X$ has a continuous transition density function $p(t, x, y)$ with respect to the measure $m_p$. 
  Similar to that for Proposition \ref{NashIneq},  Propositions \ref{P:3.2} and \ref{offdiagUBE}, 
using the classical Nash's inequality for one- and two-dimensional Brownian motion
and Davies method, one can easily establish the following.

\begin{prop}\label{UB-arch}
Let $T\geq 2$. 
There exist $C_1, C_2>0$ such that
\begin{displaymath}
p (t,x,y)\le C_1\left(\frac{1}{t}+
\frac{1}{t^{1/2}}\right)e^{-C_2\rho(x,y)^2/t}  \qquad\text{for all } x, y \in
E, t\in (0,T].
\end{displaymath}
\end{prop}

The next theorem gives short time sharp two-sided estimates on $p(t, x, y)$.

\begin{thm}\label{T:smalltime-arch}
 Let $T\geq 2$ be fixed. There exist positive constants $C_i$, $3\le i\le 16$,
  so  that the transition density $p(t,x,y)$ of BMVD $X$ on $E$ satisfies the following estimates when $t\in (0, T]$:
\begin{description}
\item{\rm (i)} For $x \in L$ and $y\in E$, 
\begin{equation}\label{eq-arch-1}
\frac{C_3}{\sqrt{t}}e^{-C_4\rho (x, y)^2/t }  \le p(t,x,y)\le\frac{C_5}{\sqrt{t}}e^{- C_6\rho (x, y)^2/t}.
\end{equation}

\item{\rm (ii)} For $x,y\in D_0\cup \{a^*_1, a^*_2\}$, when $\rho(x, a^*_i)+\rho(y, a^*_i)<1$ for some $i=1, 2$,
\begin{eqnarray}\label{eq-arch-3}
&& \frac{C_{7}}{\sqrt{t}}e^{- C_{8}\rho(x,y)^2/t}+\frac{C_{7}}{t}\left(1\wedge
\frac{\rho(x, a^*_i)}{\sqrt{t}}\right)\left(1\wedge
\frac{\rho(y, a^*_i)}{\sqrt{t}}\right)e^{- C_{9}|x-y|^2/t} \\
 &\le  & p(t,x,y)
\le \frac{C_{10}}{\sqrt{t}}e^{-C_{11}\rho(x,y)^2/t}
+\frac{C_{10}}{t}\left(1\wedge
\frac{\rho(x, a^*_i)}{\sqrt{t}}\right)\left(1\wedge
\frac{\rho(y,a^*_i)}{\sqrt{t}}\right)e^{-C_{12}|x-y|^2/t}; \nonumber
\end{eqnarray}
otherwise,
\begin{equation}\label{eq-arch-4}
\frac{C_{13}}{t}e^{- C_{14}\rho(x,y)^2/t} \le p(t,x,y) \le \frac{C_{15}}{t}e^{- C_{16}\rho(x,y)^2/t} . 
\end{equation}
 \end{description}
\end{thm}

\begin{proof} This theorem can be established by a similar consideration as that for
Theorem \ref{small-time-multi-line}.
Here we only give a  brief sketch. 
Let $C_1>0$ and $C_2\in (0, 1/2)$ be the constants in Lemma \ref{L:6.1}.

\smallskip

\noindent {\it Case 1.}  $\rho(x,y)<2C_2$.  The desired estimates can be obtained in a similar way
as that for {\it Case 1} in the proof of Theorem \ref{small-time-multi-line}, by using Lemma \ref{L:6.1}. 

\smallskip

\noindent {\it Case 2. } $\rho(x,y)\geq 2C_2$. 
Due to the upper bound estimate in Theorem \ref{UB-arch}, it suffices to 
show the following lower bound estimate hold: 
there exists some $c_1, c_2 >0$, such that
\begin{equation}\label{e:8.11}
p (t,x,y)\ge c_1 e^{-c_2 \rho(x,y)^2/t}  \quad \text{for all }t\in (0, T].
\end{equation}
We divided its proof  into three cases. 

\begin{description}
\item{(i)}  Both $x$ and $y$ are in $ L\cup B_\rho (a^*_1, C_2) \cup
B_\rho (a^*_2, C_2)$.  Without loss of generality, we assume $x\in L\cup B_\rho (a^*_1, C_2)$ and $x$ is closer
to $a^*_1$ if both $x$ and $ y$ are on the arch $L$.
Denote by $\ell (w, z)$ the arch length in $L$ between two points $w, z\in L$,
and $D_1:= L\cup B_\rho (a^*_1, 3C_2) \cup
B_\rho (a^*_2, 3C_2)$.  
We define a modified signed radial process $Y_t=u(X^{D_1}_t)$,
where 
$$
u(z):=\begin{cases}
\rho (z, a^*_1) & \hbox{if  } z \in D_0\cap B_\rho (a^*_1, 3C_2) ,  \cr
- \ell (a^*_1, z) & \hbox{if  } z \in L, \cr
 -\ell (a^*_1, a^*_2) -\rho (z , a^*_2)    & \hbox{if  } z  \in  D_0\cap B_\rho (a^*_2, 3C_2).
 \end{cases}
 $$ 
By a similar argument as that for Section \ref{S:4}, we can show that 
 $Y$  is a subprocess of a skew Brownian motion on $\IR$ with skewness at points  0  and $\ell (a^*_1, a^*_2)$ and   bounded drift 
  killed upon leaving $I=(- \ell (a^*_1,a^*_2)-2C_2, 2C_2)$. The desired lower bound estimate 
  for $p(t, x, y)$ can be
  derived from the Dirichlet heat kernel estimate for the one-dimensional diffusion $Y$. 

\item{(ii)} Both  $x, y \in D_2:=D_0\setminus (B_\rho(a^*_1, C_2/2)\cup B_\rho(a^*_2, C_2/2))$. 
In this case, the desired lower bound estimate for $p(t, x, y)$ follows from the Dirichlet heat kernel estimate 
for two-dimensional Brownian motion in $C^{1,1}$ domain $D_2$. 

\item{(iii)} $x\in L \cup B_\rho(a^*_1, C_2/2)\cup B_\rho(a^*_2, C_2/2)$ and $y\in D_0\setminus (B_\rho(a^*_1, 2C_2)\cup B_\rho(a^*_2, 2C_2))$. Without loss of generality, we assume $x$ is closer to $a^*_1$ than to $a_2^*$. Let 
$$ D_3:=  \{z\in D_0 : C_2/2\le \rho(z, a^*_1)\le  C_2\}.
$$
Note that $\rho(x,y)\ge  (\rho (x,z)+\rho(y, z))/5$ for $z\in D_3$. 
 By Markov property,
$$
p(t,x,y)\ge \int_{D_3} p(t/2, x, z)p(t/2, z,y)m_p(dz).
$$
 The desired lower bound for $p(t,x,y)$   follows from  the results obtained in (i) and (ii). 
\end{description}
\end{proof}

\vskip 0.3truein

\noindent {\bf Zhen-Qing Chen}

\smallskip \noindent
Department of Mathematics, University of Washington, Seattle,
WA 98195, USA

\noindent
E-mail: \texttt{zqchen@uw.edu}

\bigskip

\noindent {\bf Shuwen Lou}

\smallskip \noindent
Department of Mathematics, Statistics, and Computer Science,

\noindent
University of Illinois at Chicago,
Chicago, IL 60607, USA

\noindent
E-mail:  \texttt{slou@uic.edu}

 \end{document}